\documentclass[leqno,a4paper,11pt]{article}
\usepackage[all]{xy}
\usepackage{amsmath}
\usepackage{amssymb}
\usepackage{amsxtra}
\usepackage{amsthm}
\usepackage[margin=3cm]{geometry}
\usepackage[breaklinks]{hyperref}

\theoremstyle{definition}
\newtheorem{ssect}[subsection]{}
\newtheorem{sssect}[subsubsection]{}
\newtheorem{Definition}[subsection]{Definition}
\newtheorem{Example}[subsection]{Example}
\newtheorem{Remark}[subsection]{Remark}
\newtheorem*{sRemark}{Remark}
\theoremstyle{plain}
\newtheorem{Lemma}[subsection]{Lemma}
\newtheorem{Proposition}[subsection]{Proposition}
\newtheorem{Theorem}[subsection]{Theorem}
\newtheorem{Corollary}[subsection]{Corollary}
\makeatletter
\newcommand{\sigle}[1]{\protect\@sigle{#1}}
\newcommand{\@sigle}[1]{\gdef\@sigleparam{#1}\aftergroup\@absorbtwo}
\newcommand{\@absorbtwo}[2]{\ifcat a#1#1#2\else \@sigleparam \fi}
\makeatother

\newcommand\etab{\bar \eta}
\newcommand\fS{\mathfrak{S}}
\newcommand\fp{\mathfrak{p}}
\newcommand\fq{\mathfrak{q}}
\newcommand\ET{\mathrm{ET}}
\newcommand\SEC{\mathrm{SEC}}
\newcommand\CH{\mathrm{CH}}
\newcommand\rig{\mathrm{rig}}
\newcommand\perm{\mathrm{perm}}
\newcommand\hperm{\text{-}\perm}
\newcommand\Kt{K\sptilde}
\newcommand\xt{x\sptilde}
\newcommand\St{S\sptilde}
\newcommand\resp{resp.\ }
\DeclareMathOperator{\Spec}{Spec}
\DeclareMathOperator{\Tr}{Tr}
\DeclareMathOperator{\Gr}{Gr}
\DeclareMathOperator{\rk}{rk}
\DeclareMathOperator{\Aut}{Aut}
\DeclareMathOperator{\Coker}{Coker}
\DeclareMathOperator{\Ker}{Ker}

\DeclareMathOperator{\Img}{Im}
\newcommand{\cf}{cf.\ }
\newcommand{\ie}{i.e.\ }
\newcommand{\loccit}{\emph{loc.\ cit.}}
\newcommand{\Dbc}{D^b_c}
\newcommand{\Qlb}{\overline{\Q_\ell}}
\newcommand{\et}{{\mathrm{et}}}
\newcommand{\cd}{\mathrm{cd}}
\newcommand{\cl}{\mathrm{cl}}
\newcommand{\Id}{\mathrm{Id}}
\newcommand{\trBr}{\Tr^{\mathrm{Br}}}
\newcommand{\Gal}{\mathrm{Gal}}
\newcommand{\Reg}{\mathrm{Reg}}
\newcommand{\red}{\mathrm{red}}
\newcommand{\beq}{\begin{equation}}
\newcommand{\eeq}{\end{equation}}
\newcommand{\longsimto}{\stackrel{\sim}{\longrightarrow}}

\newcommand{\Hom}{\mathrm{Hom}}
\newcommand{\Supp}{\mathrm{Supp}}
\newcommand{\Suppc}{\overline{\mathrm{Supp}}}

\newcommand{\cosk}{\mathrm{cosk}}

\newcommand{\lisse}{\mathrm{lisse}}
\newcommand\cB{{\mathcal B}}
\newcommand\cC{{\mathcal C}}
\newcommand\cD{{\mathcal D}}
\newcommand\cE{{\mathcal E}}
\newcommand\cF{{\mathcal F}}
\newcommand\cG{{\mathcal G}}
\newcommand\cO{{\mathcal O}}
\newcommand\cS{{\mathcal S}}

\newcommand\cX{{\mathcal X}}
\newcommand\cY{{\mathcal Y}}

\newcommand\fH{{\mathcal H}}

\newcommand\F{\mathbb{F}}

\newcommand\Z{{\mathbb Z}}
\newcommand\Q{{\mathbb Q}}
\newcommand\C{{\mathbb C}}

\begin{document}
\title{Odds and Ends on Finite Group Actions and Traces}
\author{Luc Illusie\thanks{Universit\'{e} de Paris-Sud,
D\'{e}partement de math\'{e}matique, B\^{a}t.\ 425, 91405 Orsay
Cedex, France; \texttt{Luc.Illusie@math.u-psud.fr}} \ and Weizhe
Zheng\thanks{Columbia University, Department of Mathematics, 2990
Broadway, New York, NY 10027, USA;
\texttt{zheng@math.columbia.edu}}}
\date{}
\maketitle

\begin{abstract}
In this article, we study several problems related to virtual traces
for finite group actions on schemes of finite type over an
algebraically closed field. We also discuss applications to fixed
point sets. Our results generalize previous results obtained by
Deligne, Laumon, Serre and others.
\end{abstract}

\renewcommand\contentsname{Summary}
\tableofcontents

\setcounter{section}{-1}
\section{Introduction}
\numberwithin{equation}{section}
Let $k$ be an algebraically closed field of characteristic~$p$, and let $X$ be a $k$-scheme, separated and of finite type,
endowed with an action of a finite group~$G$. If $\ell$ is a prime
number $\ne p$, $G$ acts on $H^*(X,{\mathbb{Q}}_{\ell})$ (resp.
$H^*_{c}(X,{\mathbb{Q}}_{\ell})$), and, for $s \in G$,  we can consider
the virtual traces
\begin{align}
t_{\ell}(s) &:= \sum (-1)^i\Tr(s,H^i(X,{\mathbb{Q}}_{\ell})),
\label{(0.1)}\\
t_{c,\ell}(s) &:= \sum (-1)^i\Tr(s,H^i_c(X,{\mathbb{Q}}_{\ell})).
\label{(0.2)}
\end{align}
These are $\ell$-adic integers. Several natural questions arise:
\begin{enumerate}
\item Is $t_{\ell}(s)$ (resp. $t_{c,\ell}(s)$) an integer independent of
$\ell$?

\item Do we have $t_{\ell}(s) = t_{c,\ell}(s)$?

\item Under suitable assumptions on the action of $G$ (freeness,
tameness), can one describe the virtual representation
\beq
\chi(X,G,{\mathbb{Q}}_{\ell}) = \sum (-1)^i[H^i(X,{\mathbb{Q}}_{\ell})]
\label{(0.3)}
\eeq
(and its analogue with compact supports)
in the Grothendieck group $R_{{\mathbb{Q}}_{\ell}}(G)$ of finite
dimensional ${\mathbb{Q}}_{\ell}$-representations of $G$, where $[-]$
denotes a class in $R_{{\mathbb{Q}}_{\ell}}(G)$?

\item How do the numbers $t_{\ell}(s)$ (resp. $t_{c,\ell}(s)$) compare
with the similar ones defined using other cohomology theories (rigid,
for example, if $p > 1$, or Betti, when $k = {\mathbb{C}}$)?
\end{enumerate}

These are old questions, and for some of them, partial answers were
obtained long ago. Recent work of Serre (\cite{Se2}, \cite{Se3}, \cite{Se4}) has
revived interest in them. The purpose of this paper is to collect
answers and discuss some applications.

For $s = 1$, $t_{\ell}(s)$ (resp. $t_{c,\ell}(s)$) is the
Euler-Poincar\'{e} characteristic $\chi(X,{\mathbb{Q}}_{\ell})$
(resp. $\chi_c(X,{\mathbb{Q}}_{\ell})$). By Grothendieck's trace
formula, $\chi_c(X,{\mathbb{Q}}_{\ell})$ is an integer independent
of $\ell$, and it is known, by a theorem of Laumon \cite{L2}, that
it is equal to $\chi(X,{\mathbb{Q}}_{\ell})$. In \S~1 we show that,
for all $s \in G$, $t_{\ell}(s) = t_{c,\ell}(s)$. Actually, we
establish a relative, equivariant form of Laumon's theorem. In \S~2
we generalize this to Deligne-Mumford stacks of finite type over a
regular base of dimension at most~$1$ under an additional
hypothesis. We also give an analogue for torsion coefficients.

By a theorem of Deligne-Lusztig \cite[3.3]{DL},
$t_{c,\ell}(s)$ is an integer $t(s)$ independent of~$\ell$. In \S~3 we
deduce vanishing theorems for $t(s)$, $s \ne 1$, and prove a
generalization (\ref{3.7}) of a divisibility theorem of Serre \cite[7.5]{I2}.

The vanishing theorem \ref{3.2} for free actions was shown by Deligne \cite{I2} to hold more
generally under a
certain tameness assumption. In \S~4 we consider actions
of $G$ on $X$ that are not necessarily free.
Using Vidal's groups $K(Y)^0_t$ (\cite{Vi1}, \cite{Vi2}) we define
a notion of \emph{virtual tameness} for the action of $G$, and
establish in this case a formula (\ref{4.7}) for $\chi(X,G,{\mathbb{Q}}_{\ell})$
\eqref{(0.3)} as a sum of certain induced characters. This is an algebraic analogue of
a formula of Verdier \cite{V}.

In \S~5 we consider Berthelot's rigid cohomology with compact
supports $H^*_{c,\rig}(X/K)$ (where $K$ is the fraction field of
$W(k)$) and, for $s \in G$, the corresponding virtual traces
\[
t_{c,\rig}(s) = \sum (-1)^i\Tr(s,H^i_{c,\rig}(X/K)). 
\]
We show that $t_{c,\rig}(s) = t(s)$. The proof uses de Jong's
alterations to reduce to the case where $X/k$ is projective and
smooth, in which case this equality was known.

A corollary of the vanishing theorems of \S~3 is that, if $G$ is an
$\ell$-group, then $\chi(X)$ is congruent to $\chi(X^G)$ modulo
$\ell$ (\ref{6.1}), where $X^G$ denotes the fixed point scheme. When
$X$ is mod $\ell$ acyclic, \ie $H^*(X,{\mathbb{F}}_{\ell}) =
H^0(X,{\mathbb{F}}_{\ell}) = {\mathbb{F}}_{\ell}$, one can say more:
$X^G$ is also mod $\ell$ acyclic. This is an analogue of a well known
theorem of P.~Smith \cite{Smith}. Here $\ell$ may be equal to $p$.
This analogue is established by Serre \cite[7.5]{Se3}. In \S~6 we
give a different proof, based on equivariant cohomology
$H^*_G(X,{\mathbb{F}}_{\ell})$, for $G$ cyclic of order $\ell$, in
the spirit of Borel \cite{Bo}. However, in contrast with the method
in \cite{Bo}, we give a shortcut, exploiting the graded module
structure of $H^*_G(X,{\mathbb{F}}_{\ell})$ over the graded algebra
$H^*(G,{\mathbb{F}}_{\ell})$. The key point is that the restriction
homomorphism $H^*_G(X,{\mathbb{F}}_{\ell}) \rightarrow
H^*_G(X^G,{\mathbb{F}}_{\ell})$ is injective and its cokernel has
bounded degree. This was inspired by \emph{localization theorems} of
Quillen \cite[4.2, 4.4]{Q1} and of Borel-Atiyah-Segal, cf. \cite[III
\S~2]{H}, \cite[\S~6]{GKM}. We also prove, along the same lines, that
if $X$ is a mod~$\ell$ cohomology sphere, then so is $X^G$
(\ref{c.1}). Finally, in \S~7 we prove an analogue of one of
Quillen's theorems \cite[4.2]{Q1}.

In a future paper \cite{IZ}, we will prove analogues of other main
theorems of Quillen \cite[2.1, 6.2]{Q1} on the structure of the mod
$\ell$ equivariant
cohomology ring.

\section{An equivariant form of a theorem of Laumon on
Euler-Poincar\'{e} characteristics}
\numberwithin{equation}{subsection}
\begin{ssect}\label{s.Laumon}
Fix a field $k$ of characteristic $p$, an algebraic
closure $\overline{k}$ of $k$, a prime number $\ell \ne p$, an algebraic
closure $\overline{{\mathbb{Q}}_{\ell}}$ of ${\mathbb{Q}}_{\ell}$. Let $X$ be
a $k$-scheme, separated and of finite type. If $X$ is smooth, Poincar\'e duality implies
\begin{equation}\label{e.Laumon}
  \chi(X,\Q_\ell)=\chi_c(X,\Q_\ell).
\end{equation}
It is known, by a theorem of Laumon \cite{L2}, that \eqref{e.Laumon} holds more generally without the smoothness assumption on $X$. Laumon also established a relative version of this result \cite[1.1]{L2}. We refer the reader to the abstract of \cite{L2} for a history of Laumon's theorem. In this section, we generalize the relative form of Laumon's theorem to the equivariant situation.
\end{ssect}

\begin{ssect}\label{1.1} Let $X$ be
a $k$-scheme, separated and of finite type, endowed with an action of a finite
group $G$ ($G$ acting trivially on $\Spec k$). In the sequel, unless otherwise stated, groups are supposed to act on the right.
We denote by $D^b_c(X,G,\overline{{\mathbb{Q}}_{\ell}})$ the
category of $G$-equivariant $\ell$-adic complexes defined in \cite[1.3]{Z}, and $K(X,G,\overline{{\mathbb{Q}}_{\ell}})$ the corresponding
Grothendieck group. For $L \in D^b_c(X,G,\overline{{\mathbb{Q}}_{\ell}})$,
we denote by $[L]$ its class in $K(X,G,\overline{{\mathbb{Q}}_{\ell}})$.
Following Laumon \cite{L2}, we denote by $\Kt(X,G,\overline{{\mathbb{Q}}_{\ell}})$
the quotient of $K(X,G,\overline{{\mathbb{Q}}_{\ell}})$ by the ideal
generated by the image of $[\overline{{\mathbb{Q}}_{\ell}}(1)_{\Spec k}]-1$, and by $\xt$ the image in
$\Kt(X,G,\overline{{\mathbb{Q}}_{\ell}})$ of
an element $x$ of $K(X,G,\overline{{\mathbb{Q}}_{\ell}})$.

Passing from
$K$ to $\Kt $ destroys a lot of arithmetic information. For
example, the function $x \mapsto \left(\Tr(-,x_{\overline{k}}), (s,g) \in G \times \Gal(\overline{k}/k) \mapsto \Tr(sg,x_{\overline{k}})\right)$
from $K(\Spec(k),G,\overline{{\mathbb{Q}}_{\ell}})$ to the set of continuous functions from $G \times \Gal(\overline{k}/k)$ to $\overline{{\mathbb{Q}}_{\ell}}$ does not pass to
the quotient, as the case where $G = \{1\}$ and $k = {\mathbb{F}}_q$
trivially shows. However, not all is lost. For example, if $k$ is a
local field (fraction field of an excellent henselian discrete valuation ring), with
residue field $k_0$, then the restrictions of these trace functions
to $G \times I$, where $I$ is the inertia group, pass to the quotient.

Recall \cite[1.5]{Z} that, for an equivariant map $(f,u) \colon (X,G) \rightarrow
(Y,H)$, we have exact functors
\[
Rf_* \colon D^b_c(X,G,\overline{{\mathbb{Q}}_{\ell}}) \rightarrow
D^b_c(Y,H,\overline{{\mathbb{Q}}_{\ell}}), \quad Rf_! \colon
D^b_c(X,G,\overline{{\mathbb{Q}}_{\ell}}) \rightarrow
D^b_c(Y,H,\overline{{\mathbb{Q}}_{\ell}}),
\]
inducing homomorphisms
\[
f_* \colon K(X,G,\overline{{\mathbb{Q}}_{\ell}}) \rightarrow K(Y,H,\overline{{\mathbb{Q}}_{\ell}}), \quad  f_! \colon K(X,G,\overline{{\mathbb{Q}}_{\ell}})
\rightarrow K(Y,H,\overline{{\mathbb{Q}}_{\ell}}),
\]
and
\[
f_* \colon \Kt (X,G,\overline{{\mathbb{Q}}_{\ell}}) \rightarrow
\Kt (Y,H,\overline{{\mathbb{Q}}_{\ell}}), \quad  f_! \colon \Kt (X,G,\overline{{\mathbb{Q}}_{\ell}})
\rightarrow \Kt (Y,H,\overline{{\mathbb{Q}}_{\ell}})
\]
by passing to the quotients.
\end{ssect}

The following is a generalization of Laumon's theorem \cite[1.1]{L2}:

\begin{Theorem} \label{1.2}
Let $(f,u) \colon (X,G) \rightarrow (Y,H)$ be an equivariant map between
$k$-schemes separated and of finite type, endowed with finite group
actions. Then, for any $x \in  K(X,G,\overline{{\mathbb{Q}}_{\ell}})$, we
have
\[
f_*(\xt ) = f_!(\xt )
\]
in $\Kt (Y,H,\overline{{\mathbb{Q}}_{\ell}})$.
\end{Theorem}

In particular, taking $Y = \Spec k$, $f \colon X \rightarrow Y$ the structural morphism, $H = G$, $u =
\Id$, we get:

\begin{Corollary}\label{1.3}
Assume $k$ algebraically closed. Let $\chi_c(X,G,{\mathbb{Q}}_{\ell})$
(resp. $\chi_c(X,G,{\mathbb{Q}}_{\ell})$) be the
image of $R\Gamma_c(X,{\mathbb{Z}}_{\ell})$ (resp. $R\Gamma(X,{\mathbb{Z}}_{\ell})$) in the Grothendieck group
$R_{{\mathbb{Q}}_{\ell}}(G) = K(\Spec k,G,{\mathbb{Q}}_{\ell})$ of
finite dimensional ${\mathbb{Q}}_{\ell}$-representations of $G$, \ie
\[
\chi_c(X,G,{\mathbb{Q}}_{\ell}) = \sum (-1)^i[H^i_c(X,{\mathbb{Q}}_{\ell})]
\]
(resp.
\[
\chi(X,G,{\mathbb{Q}}_{\ell}) = \sum (-1)^i[H^i(X,{\mathbb{Q}}_{\ell})]\text{).}
\]
Then we have
\beq
\chi_c(X,G,{\mathbb{Q}}_{\ell}) = \chi(X,G,{\mathbb{Q}}_{\ell}). \label{(1.3.1)}
\eeq
In other words, with the notations of \eqref{(0.1)} and \eqref{(0.2)}, for $s \in G$, we have:
\beq
t_{\ell}(s) = t_{c,\ell}(s). \label{(1.3.2)}
\eeq
\end{Corollary}

For $Y = \Spec k$ and  $H = \{1\}$, we get:
\begin{Corollary}\label{1.4}
Assume $k$ algebraically closed. Set
\begin{align*}
\chi_G(X,{\mathbb{Q}}_{\ell}) &= \sum (-1)^i\dim H^i_G(X,{\mathbb{Q}}_{\ell}),
\\
\chi_{c,G}(X,{\mathbb{Q}}_{\ell}) &= \sum (-1)^i\dim H^i_{c,G}(X,{\mathbb{Q}}_{\ell}),
\end{align*}
where $H^i_G(X,{\mathbb{Q}}_{\ell}) = R^if_*{\mathbb{Q}}_{\ell}$ (resp.
$H^i_{c,G}(X,{\mathbb{Q}}_{\ell}) = R^if_!{\mathbb{Q}}_{\ell}$) is the
equivariant cohomology of $X/k$ with no supports (resp. with compact
supports). Then we have:
\beq
\chi_{c,G}(X,{\mathbb{Q}}_{\ell}) = \chi_G(X,{\mathbb{Q}}_{\ell}).  \label{(1.4.1)}
\eeq
\end{Corollary}

By definition, the number in \eqref{(1.4.1)} is
the Euler-Poincar\'{e} characteristic of the Deligne-Mumford stack
$[X/G]$.

\begin{Remark}\label{d.adm}
(a) If $X$ is smooth, Poincar\'e duality implies
$t_\ell(s)=t_{c,\ell}(s^{-1})$. Thus, in this case, \eqref{(1.3.2)}
follows from the fact that $t_{c,\ell}(s)$ is an integer, which is a
result of Deligne and Lusztig \cite[3.3]{DL}. See \ref{3.1}.

(b) Recall that the action of $G$ on $X$ is \emph{admissible} if $X$
is a union of $G$-stable open affine subschemes (cf. \cite[V
1]{SGA1}), which implies that $X/G$ exists as a scheme, is separated
and of finite type, and the projection $\pi \colon X \rightarrow X/G$
is finite. In this case, \eqref{(1.4.1)} brings no new information as
it boils down to the original form of Laumon's theorem. Indeed, one
has
\[
\chi_c([X/G],\overline{{\mathbb{Q}}_{\ell}}) = \chi_c(X/G,\overline{{\mathbb{Q}}_{\ell}})
\]
and similarly with $\chi$ (cf. \cite[1.7 (a)]{Z}).
\end{Remark}

Finally, here is an application to local fields, coming from \cite[4.2]{Vi1}:
\begin{Corollary}\label{1.5}
Let $K$ be the fraction field of an excellent henselian discrete valuation ring of residue field~$k$,
$\overline{K}$ an algebraic closure of $K$. Let
$\eta_1$ be a finite, normal extension of $\eta = \Spec K$,
with $\kappa(\eta_1)$ contained in~$\overline{K}$, and
$X/\eta_1$ a scheme separated and of finite type. We assume that a
finite group $G$ acts on $X \rightarrow \eta_1$ by
$\eta$-automorphisms. Let $I$ be the inertia
group of~$\eta$. Then, for $(s,g) \in G \times I$ such that $s$ and
$g$ induce the same automorphism of $\eta_1/\eta$,
\beq
\Tr((s,g), R\Gamma(X_{\overline{K}},{\mathbb{Q}}_{\ell})) =\Tr((s,g), R\Gamma_c(X_{\overline{K}},{\mathbb{Q}}_{\ell})), \label{(1.5.1)}
\eeq
and this $\ell$-adic number is an integer independent of $\ell$.
\end{Corollary}

By base change to the maximal unramified extension of $K$, we may
assume the residue field separably closed. We apply 1.2 to the
projection $f \colon X
\rightarrow \eta_1$ and $1 \in K(X,G,{\mathbb{Q}}_{\ell})$, observing
that the group
$K(\eta_1,G,{\mathbb{Q}}_{\ell})$ ($= \Kt (\eta_1,G,{\mathbb{Q}}_{\ell})$) can be identified with the Grothendieck group $R_{{\mathbb{Q}}_{\ell}}(\Gamma)$ of continuous ${\mathbb{Q}}_{\ell}$-linear representations of
$\Gamma = G \times_{\Aut(\eta_1/\eta)} \Gal(\overline{K}/K)$. We get \eqref{(1.5.1)}. The last assertion is proven
in \cite[4.2]{Vi1} when the action of $G$ on $X$ is admissible. The general case follows by induction since $X$ has an affine $G$-stable dense open subset. When the residue field is finite, it also follows from
\cite[1.16]{Z}.

\begin{proof}[Proof of \ref{1.2}]
Factorizing $(f,u)$ into
\[ (X,G)
\xrightarrow{(f,\Id_G)} (Y,G)
\xrightarrow{(\mathrm{Id}_Y,u)} (Y,H)
\]
and applying the formula $R(\Id_Y,u)_*\simeq R(\Id_Y,u)_!$ \cite[1.5
(ii), (iii)]{Z}, we may assume $G = H$, $u = \Id$. By Nagata's
compactification theorem (\cite{Co}, \cite{Lu}) and de Jong's
construction in \cite[7.6]{dJ}, we can find (cf. \cite[3.7]{Z}) a
$G$-equivariant compactification $f = gj$, where $g \colon Z
\rightarrow Y$ is proper and $j \colon X \rightarrow Z$ is a dense
open immersion. Let $i \colon Z-X \rightarrow Z$ be a complementary
closed immersion. As in Laumon's proof, we are then reduced to
showing that for any $x \in K(X,G,\overline{{\mathbb{Q}}_{\ell}})$,
the image of $i^*j_*x$ in $\Kt
(Z-X,G,\overline{{\mathbb{Q}}_{\ell}})$ is zero. Therefore, changing
notations, to prove 1.2 it suffices to establish the following
(equivalent) result:
\end{proof}

\begin{Proposition}\label{1.6}
Let $X$ be a $k$-scheme separated and of finite type, endowed with an action of
the finite group $G$, $i \colon Y \rightarrow X$ an equivariant closed
immersion, $j \colon U = X - Y \rightarrow X$ the (equivariant) complementary open
immersion. Then, for any $x \in K(U,G,\overline{{\mathbb{Q}}_{\ell}})$, we
have
\[
i^*j_*(\xt ) = 0
\]
in $\Kt (Y,G,\overline{{\mathbb{Q}}_{\ell}})$.
\end{Proposition}

To prove \ref{1.6}, we start by imitating Laumon's reduction to
\cite[2.2.1]{L2}. Let $f\colon X'\to X$ be the blow-up of $Y$ and
consider the following commutative diagram with Cartesian square
\[\xymatrix{& X'\ar[d]^f & Y'\ar[l]_{i'}\ar[d]^{g}\\
U\ar[r]_j\ar[ru]^{j'} & X& Y\ar[l]^{i}}
\]
By proper base change,
\[i^*j_*x=i^*f_*j'_*x=g_*i'^*j'_* x.\]
We are
thus reduced to the case where $Y$ is a Cartier divisor. By
cohomological descent for a finite covering of $X$ by $G$-stable
open subsets, we may further assume that $Y$ is defined by a global
equation $F \in \Gamma(X,{\cal O}_{X})$. Up to replacing $F$ by
$\prod_{s \in G} s^*F$ (and changing the scheme structure of $Y$),
we may assume that $F$ is invariant under $G$. Then $F$ defines a
$G$-equivariant morphism $f \colon X \rightarrow {\mathbb{A}}^1_k$,
$G$ acting trivially on ${\mathbb{A}}^1_k$, with fiber $Y$ at
$\{0\}$. As in (\loccit), replacing ${\mathbb{A}}^1_k$ by its
henselization $S$ at $0$, with closed point $s$ and generic point
$\eta$, we are reduced to showing that, for any $K \in
D^b_c(X_{\eta},G,\overline{{\mathbb{Q}}}_{\ell})$, the class in $\Kt
(Y,G,\overline{{\mathbb{Q}}_{\ell}})$ of
\[
R\Gamma(I,R\Psi K) \in D^b_c(Y,G,\overline{{\mathbb{Q}}_{\ell}})
\]
is zero, where $I$ is the inertia subgroup of
$\Gal(\overline{\eta}/\eta)$, $\overline{\eta}$ a separable closure
of $\eta$. It then suffices to invoke the following analogue of
\cite[2.2.1]{L2}.

\begin{Theorem}\label{1.7}
Let $S$ be the spectrum of a henselian discrete valuation ring, with
closed point~$s=\Spec k$ and generic point $\eta$, $Y$ be a scheme
of finite type over $s$, endowed with an action of a finite group
$G$ ($G$ acting trivially on~$s$). Then for any $L\in \Dbc(Y\times_s
\eta, G, \Qlb)$, the class in $\Kt (Y,G,\Qlb)$ of \[R\Gamma(I,L)\in
\Dbc(Y,G,\Qlb)\] is zero, where $I$ is the inertia subgroup of
$\Gal(\overline{\eta}/\eta)$, $\overline{\eta}$ a separable closure
of $\eta$.
\end{Theorem}

We refer to \cite[4.1]{Z} for the definition of $\Dbc(Y\times_s
\eta, G, \Qlb)$. In our case, it is based on the topos
$(Y\times_s\eta,G)\sptilde$ consisting of sheaves on
$Y_{\overline{s}}$ endowed with an action of
$\Gal(\overline{\eta}/\eta)\times G$ compatible with the action of
$\Gal(\overline{s}/s)$ on $Y_{\overline{s}/s}$. Here $\overline{s}$
is a separable closure of $s$.

\begin{proof}
Let $P_{\ell}$ be the kernel of the $\ell$-component $t_{\ell}
\colon I \rightarrow I_{\ell} ={\mathbb{Z}}_{\ell}(1)$ of the tame
character, and let $\eta_\ell=\overline{\eta}/P_\ell$. We consider
the topos of $G$-$\Gal(\eta_\ell/\eta)$-sheaves on
$Y_{\overline{s}}$, which is the subcategory of $(Y\times_s
\eta,G)\sptilde$ consisting of sheaves on which $P_\ell$ acts
trivially. Since
\[
R\Gamma(I,L) = R\Gamma(I_{\ell},L^{P_{\ell}}),
\]
one is reduced to showing that, for any
$G$-$\Gal(\eta_\ell/\eta)$-$\Qlb$-sheaf $L$ on $Y_{\overline{s}}$,
the class of $R\Gamma(I_{\ell},L) $ in $\Kt
(Y,G,\overline{{\mathbb{Q}}_{\ell}})$ is zero, where $\eta_{\ell} =
\overline{\eta}/P_{\ell}$. We use the argument of Deligne at the end of \cite{L2}. If
$\sigma$ is a topological generator of $I_{\ell}$, we have a
decomposition
\begin{equation}\label{1.7.1}
L = \bigoplus_{\alpha \in \overline{{\mathbb{Q}}_{\ell}}^\times}\bigcup_{n\ge
1}\Ker((\sigma -\alpha)^n,L).
\end{equation}
Let $L^u$ be the largest subsheaf of $L$ (in the category of
$G$-$\Gal(\eta_{\ell}/\eta)$-$\overline{{\mathbb{Q}}_{\ell}}$-sheaves
on $Y_{\overline{s}}$) on which the
action of $I_{\ell}$ is unipotent. In terms of the above decomposition,
\begin{equation}\label{1.7.2}
L^u = \bigcup_{n\ge 1}\Ker((\sigma-1)^n,L).
\end{equation}
As the formation of $L^u$ and of $R\Gamma(I_{\ell},-)$ commutes with taking
stalks at geometric points of $Y$, the inclusion $L^u \rightarrow L$
induces an isomorphism
\[
R\Gamma(I_{\ell},L^u) \longsimto R\Gamma(I_{\ell},L).
\]
Therefore we may assume that the action of $I_{\ell}$ on $L$ is
unipotent. Thus there exists a (twisted) nilpotent endomorphism $N \colon L
\rightarrow L(-1)$ (a morphism of $G$-$\Gal(\eta_{\ell}/\eta)$-$\overline{{\mathbb{Q}}_{\ell}}$-sheaves on $Y_{\overline{s}}$) such that the
representation $\rho \colon I_{\ell} \rightarrow \Aut(L)$ is given by $\rho(g) = \exp(Nt_{\ell}(g))$ for $g
\in I_{\ell}$. By definition, $H^0(I_{\ell},L) = L^{I_{\ell}}$, and
we have a
canonical isomorphism
\[
H^1(I_{\ell},L) \longsimto (L_{I_{\ell}})(-1),
\]
where $L_{I_{\ell}}$ is the sheaf of co-invariants of $I_{\ell}$ in
$L$ (given by $z \mapsto [z(\sigma)] \otimes \sigma\spcheck$ on
1-cocycles of the standard cochain complex, where $\sigma$ is a
generator of $I_{\ell} = {\mathbb{Z}}_{\ell}(1)$, $\sigma\spcheck
\in {\mathbb{Z}}_{\ell}(-1)$ its dual, and $[-]$ means a class in
$L_{I_{\ell}}$). In other words,
\[
H^0(I_\ell,L)\simeq \Ker N, \quad H^1(I_\ell, L)\simeq \Coker N.
\]
Moreover, $H^i(I_\ell,L)=0$ for $i\neq 0,1$. Using the monodromy
filtration \cite[1.6.14]{WeilII} $\dots \subset M_iL \subset
M_{i+1}L \subset \dotsb$ of $L$ in the category of
$G$-$\Gal(\eta_\ell/\eta)$-$\overline{{\mathbb{Q}}_{\ell}}$-sheaves
on $Y_{\overline{s}}$, one gets the isomorphisms
\[
\Gr^M_i((\Coker N)(1)) (i) \longsimto \Gr^M_{-i}(\Ker N),
\]
which implies that $\Ker N$ and $\Coker N$ have the same image in
$\Kt(I_\ell,\Qlb)$.
\end{proof}

The analogue of \ref{1.7} (and, in turn, \ref{1.6} and \ref{1.2}) with $\Qlb$ replaced by an algebraic extension~$E$ of a complete discrete valuation field of characteristic $(0, \ell)$ still holds. In fact, in this case, although we may no longer have a decomposition of $L$ as in \eqref{1.7.1}, the expression for $L^u$ still holds.

\begin{ssect}\label{1.8} With the notations and hypotheses of \ref{1.6}, assume $k$
\emph{algebraically closed} and $X$ \emph{proper}. Let $L \in D^b_c(U,G,\overline{{\mathbb{Q}}_{\ell}})$, and $s
\in G$ such that the fixed point set $X^s$ of $s$ is
contained in $Y$. By \ref{1.2} we know that
\beq
\Tr(s,R\Gamma_c(U,L)) = \Tr(s,R\Gamma(U,L)).
\label{(1.8.1)}
\eeq
Though $s$ has no fixed points on $U$, this trace can be nonzero
(for example, if $X$ is the affine line over $k$, with $p >
1$, and $G$ the cyclic group ${\mathbb{Z}}/p{\mathbb{Z}}$ acting on $X$ by
translation, with generator $s \colon x \mapsto x +1$, then $t_c(s) = 1$).
We can rewrite both sides as
\begin{align*}
\Tr(s,R\Gamma_c(U,L)) &= \Tr(s,R\Gamma(X,j_!L)),
\\
\Tr(s,R\Gamma(U,L)) &= \Tr(s,R\Gamma(X,Rj_*L)).
\end{align*}
By the Lefschetz-Verdier trace formula \cite[III 4.7]{SGA5}, each of these
traces is a sum of ``local terms at infinity'', associated with the
connected components of $X^s \subset Y$:
\begin{align*}
\Tr(s,R\Gamma(X,j_!L)) &= \sum_{Z \in \pi_0(X^s)} \Tr(s,R\Gamma(X,j_!L))_Z,
\\
\Tr(s,R\Gamma(X,Rj_*L)) &= \sum_{Z \in \pi_0(X^s)} \Tr(s,R\Gamma(X,Rj_*L))_Z,
\end{align*}
where the subscript $Z$ means the local Verdier term at $Z$ for the
correspondence defined by $s \colon X \rightarrow X$ and $s^* \colon j_!L \rightarrow
j_!L$ or $s^* \colon Rj_*L \rightarrow Rj_*L$. We have the following
refinement of \eqref{(1.8.1)}:
\end{ssect}

\begin{Corollary}\label{1.9}
With the notations and hypotheses of \ref{1.8}, for each $Z \in \pi_0(X^s)$,
we have
\[
\Tr(s,R\Gamma(X,j_!L))_Z = \Tr(s,R\Gamma(X,Rj_*L))_Z.
\]
\end{Corollary}

By the additivity of $\Tr(s,-)_Z$, we have
\[
\Tr(s,R\Gamma(X,j_!L))_Z - \Tr(s,R\Gamma(X,Rj_*L))_Z = \Tr(s,R\Gamma(X,i_*i^*Rj_*L))_Z.
\]
By the (trivial) Lefschetz-Verdier formula for $i \colon Y \rightarrow X$
\cite[III 4.4]{SGA5},
\[
\Tr(s,R\Gamma(X,i_*i^*Rj_*L))_Z = \Tr(s,R\Gamma(Y,i^*Rj_*L))_Z.
\]
By definition \cite[III 4.7]{SGA5}, if $\Id_E$ denotes the identity
correspondence on $E = i^*Rj_*L$,
\[
\Tr(s,R\Gamma(Y,i^*Rj_*L))_Z = (a_Z)_*\langle s,\Id_E\rangle_Z,
\]
where $a_Z \colon Z \rightarrow \Spec k$ is the projection,
$\langle s,\Id_E\rangle_Z \in
H^0(Z,K_Z)$
is the Verdier term at $Z$ for the
correspondences $s$ and $\Id_E$ \cite[III (4.2.7)]{SGA5}, $K_Z = Ra_Z^!\overline{{\mathbb{Q}}_{\ell}}$, and
$(a_Z)_* \colon
H^0(Z,K_Z) \rightarrow \overline{{\mathbb{Q}}_{\ell}}$ is the trace map,
defined by the adjunction map $Ra_*K_Z \rightarrow \overline{{\mathbb{Q}}_{\ell}}$. More generally, for any $F \in D^b_c(Y,G,\overline{{\mathbb{Q}}_{\ell}})$ and $Z \in \pi_0(X^s)$, we have a Verdier term
$\langle s,\Id_F\rangle_Z
\in H^0(Z,K_Z)$. By \cite[III (4.13.1)]{SGA5}, $\langle s,\Id_F\rangle_Z$
is additive in
$F$, hence depends only on the class of $F$ in $K(Y,G,\overline{{\mathbb{Q}}_{\ell}})$. Therefore, by \ref{1.6}, we have $\langle s,\Id_E\rangle_Z = 0$, which
completes the proof.

\begin{ssect}
  Let us mention some analogues of Laumon's theorem in topology. Let $X$ be an $n$-dimensional manifold. If Poincar\'e duality holds for $X$, then we have the following analogue of \ref{s.Laumon}:
  \[
    \chi(X,\Q)=(-1)^n\chi_c(X,\Q).
  \]
  In particular, odd-dimensional compact manifolds have vanishing Euler-Poincar\'e characteristic. More generally, Sullivan \cite{Sullivan} has shown that compact stratified spaces (in the sense of Thom) with odd-dimensional strata have vanishing Euler-Poincar\'e characteristic. Weinberger, Goresky and MacPherson used this to show that $\chi(X,\Q)=\chi_c(X,\Q)$ holds for all stratified spaces $X$ with even-dimensional strata. See \cite[p.~141, Note~13]{Fulton}.
\end{ssect}

\section{A generalization to Deligne-Mumford stacks}

In the situation of \ref{1.1}, we have
\[
D^b_c(X,G,\overline{{\mathbb{Q}}_{\ell}}) \simeq D^b_c([X/G],\overline{{\mathbb{Q}}_{\ell}}),
\]
where $[X/G]$ denotes the Deligne-Mumford stack associated to the
action of $G$ on $X$. Moreover, the equivariant operations $Rf_*$,
$Rf_!$ correspond to similar operations for the associated
morphisms of Deligne-Mumford stacks. In the first half of this section we show that \ref{1.2}
extends to
morphisms of Deligne-Mumford stacks of finite type over a regular base of dimension $\le 1$ satisfying the condition (A) below.
In the second half, we establish an analogue for torsion coefficients.
The results of this section will not be used in the following ones with the exception of \ref{s.space}, where only the extension of \ref{1.2} to algebraic spaces is used.

\begin{ssect}\label{2.1} In this section, unless otherwise stated, we fix a (Noetherian) regular base scheme $S$ of dimension $\le 1$ satisfying the condition
\begin{itemize}
  \item[(A)] Every nonempty scheme of finite type over $S$ has a nonempty geometrically unibranch \cite[6.15.1]{EGAIV} open subscheme.
\end{itemize}
This condition is satisfied if $S$ is a Nagata scheme
\cite[033S]{Stacks} (by \cite[9.7.10]{EGAIV}) or if $S$ is
semi-local. We fix a prime number $\ell$ invertible on $S$ and an
algebraic extension $E$ of a complete discrete valuation field $E_0$
of characteristic $(0, \ell)$. We use the convention of
\cite[4.1]{LMB} for Deligne-Mumford stacks. In particular, the
diagonal of a Deligne-Mumford stack is assumed to be quasi-compact
and separated. For a Deligne-Mumford $S$-stack $\cX$ of finite type,
we denote by $\Dbc(\cX,E)$ the category of bounded $E$-complexes. If
$S$ is affine and excellent and if all schemes of finite type over
$S$ has finite $\ell$-cohomological dimension, the construction of
$D^b_c(\cX,E)$ and of the corresponding six operations is done in
\cite{LO2}. For the general case, see \cite{Zhengl}. Note that the
following sections do no depend on \cite{Zhengl} because there we
work over an algebraically closed field.

We denote by $K(\cX, E)$ the Grothendieck group of $D^b_c(X,E)$.
For $L\in \Dbc(\cX,E)$, we denote by $[L]$ its class in
$K(\cX,E)$. As in \ref{1.1}, we denote by $\Kt(\cX,E)$ the quotient
of $K(\cX,E)$ by the ideal generated by $[E(1)]-1$, and by $\xt$
the image in $\Kt(\cX, E)$ of an element $x$ of $K(\cX, E)$.

Recall that, for a morphism $f\colon \cX \to \cY$ of
Deligne-Mumford $S$-stacks of finite type, we have exact functors
\[Rf_*, Rf_! \colon \Dbc(\cX,E) \to \Dbc(\cY,E) \]
inducing homomorphisms
\[f_*, f_! \colon K(\cX,E) \to K(\cY, E) \]
and
\[f_*, f_! \colon \Kt(\cX,E) \to \Kt(\cY, E) \]
by passing to quotients.
\end{ssect}

The following is a generalization of \ref{1.2}. The proof will be given in \ref{2.5}.

\begin{Theorem}\label{2.2}
Let $f\colon \cX \to \cY$ be a morphism of Deligne-Mumford
$S$-stacks of finite type. Then, for any $x\in K(\cX,E)$, we have
\[f_*(\xt) = f_!(\xt) \]
in $\Kt(\cY,E)$.
\end{Theorem}

For a point $\xi$ of $\cY$, we denote by $i_\xi \colon \cG_\xi \to
\cY$ its residue gerbe \cite[11.1]{LMB}. It is isomorphic to a
quotient stack $[\Spec K/G]$, where $K$ is a finite type extension
field of $\kappa(s)$, $s\in S$ is the image of $\xi$, and $G$ is a
finite group acting on $K$ on the left leaving $\kappa(s)$ fixed. To
see this, we may assume, by \cite[6.1.1]{LMB}, that $\cY=[Y/H]$ is
the quotient stack of an $S$-scheme $Y$ of finite type by a finite
group $G$ acting on $Y$ on the right by $k$-automorphisms. A
representative $\Spec K_1 \to \cY$ of $\xi$ corresponds to an
$H$-torsor $T$ over $K_1$ together with an $H$-equivariant map
$t\colon T\to Y$. Let $X$ be the image of~$t$ and endow it with the
scheme structure $X=\coprod_{y\in X} \Spec \kappa(y)$. Then $t$
factorizes into $H$-equivariant maps $T\to X\to Y$. Hence $\cG_\xi
\simeq [X/H] \simeq [y/G_y]$, where $y\in X$ and $G_y < H$ is the
stabilizer of $y$. Note that $\cG_\xi$ is also the residue gerbe of
the point $\xi\to \cY_s$, where $\cY_s=\cY\times_S s$.

\begin{Lemma}\label{2.3}
The homomorphism
\[ \Kt(\cY, E) \to \prod_\xi \Kt(\cG_\xi, E) \]
induced by $i_\xi^*$, where $\xi$ runs over all points of $\cY$, is
an injection.
\end{Lemma}

\begin{proof}
We prove \ref{2.3} by Noetherian induction on $\cY$. The assertion
being trivial for $\cY=\emptyset$, we assume $\cY$ nonempty. Let $x$
be an element of $K(\cY,E)$ such that $i_\xi^*(\xt)=0$ in
$\Kt(\cG_\xi,E)$ for any $\xi$. We shall show $\xt =0$. As above, by
\cite[6.1.1]{LMB}, there exists a nonempty open immersion $j\colon
[Y/H]\hookrightarrow \cY$ from the quotient stack of an $S$-scheme
$Y$ of finite type by a finite group $H$ acting on $Y$ on the right
by $S$-automorphisms. By (A), shrinking $Y$ if necessary, we may
assume $Y$ is geometrically unibranch and connected, and $j^*x=[\cF]
- [\cF']$ for lisse $E$-sheaves $\cF$ and $\cF'$ on~$[Y/H]$. Here by
\emph{lisse $E$-sheaf} we mean a sheaf of the form
$\cE\otimes_{\cO_{E_1}} E$, where $\cO_{E_1}$ is the ring of
integers of a finite extension $E_1$ of $E_0$ contained in $E$, and
$\cE$ is a lisse $\cO_{E_1}$-sheaf. Let $i\colon \cY-[Y/H]\to \cY$
be a complementary closed immersion. Since $x=j_!j^*x+i_*i^*x$ and
$i^*(\xt)=0$ by induction hypothesis, it suffices to show
$j^*(\xt)=0$. We may thus assume $\cY=[Y/H]$, $Y$ geometrically
unibranch, $x=[\cF] - [\cF']$ where $\cF$, $\cF'$ are lisse
$E$-sheaves.

Let $\etab$ is a geometric point above the
generic point $\eta$ of $Y$. The category of lisse $E$-sheaves  on
$\cY$ under our convention is equivalent to the category of (finite-dimensional)
$E$-representations of $\pi_1(\cY, \etab)$ (see \cite[\S~4]{N} for a
definition of $\pi_1(\cY, \etab)$). Let $K_{\lisse}(\cY, E)$ denote
the corresponding Grothendieck group. The $H$-equivariant map
$\epsilon\colon \eta\to Y$ induces a morphism $\iota\colon [\eta/H]\to \cY$,
which is the residue gerbe of $\cY$ at the point $\eta\to \cY$. By
hypothesis, $\iota^*(\xt) = 0$. To see $\xt=0$, it suffices to show that
\[\iota^* \colon
\Kt_{\lisse}(\cY, E) \to \Kt([\eta/H],E)\] is an injection.

By the Jordan-H\"older theorem, $K_{\lisse}(\cY, E)$ (\resp
$K([\eta/H],E)$) is a free abelian group with the set of isomorphism
classes $S_1$ (\resp $S_2$) of simple $E$-representations of
$\pi_1(\cY,\etab)$ (\resp $\pi_1([\eta/H], \etab)$) as a base. Let
$\St_i$ be the quotient set of $S_i$ by the equivalence relation
$\sim$ defined by $L \sim M$ if there exists $n\in \Z$ such that $L
\simeq M(n)$, $i=1,2$. Then $\Kt_{\lisse}(\cY, E)$ (\resp
$\Kt([\eta/H],E)$) is a free abelian group with basis $\St_1$ (\resp
$\St_2$). We have the following morphism of short exact sequences of
groups
\[\xymatrix{1\ar[r] & \pi_1(\eta,\etab)\ar[r]\ar[d]^{\pi_1(\epsilon,\etab)} &\pi_1([\eta/H],\etab)\ar[r]\ar[d]^{\pi_1(\iota,\etab)} &
H\ar[r]\ar@{=}[d] & 1\\
1\ar[r] & \pi_1(Y,\etab)\ar[r]& \pi_1(\cY,\etab)\ar[r] & H\ar[r] & 1
}
\]
Since $Y$ is geometrically unibranch, the homomorphism
$\pi_1(\epsilon, \etab)$ is surjective. Hence the same is true for
$\pi_1(\iota, \etab)$. The latter clearly induces an injection $S_1
\to S_2$. Therefore, for all $L, M \in S_1$ satisfying $\iota^* L
\simeq (\iota^* M)(n) \simeq \iota^*(M(n))$, we have $L \simeq
M(n)$. In other words, $\iota^*$ gives an injection $\St_1 \to
\St_2$. This completes the proof of \ref{2.3}.
\end{proof}

\begin{ssect}\label{2.4.1}
The preceding lemma allows us to show that
the analogue of \ref{1.2} still holds over the base scheme $S$
($G$ and $H$ acting trivially on $S$) and with
$\Qlb$ replaced by $E$. As before, we are reduced to proving the
analogue of \ref{1.6} over $S$. Note that the case where $S$ is the spectrum of a discrete valuation ring is easy (as observed by Vidal \cite[0.1]{Vi2}).

Consider first the case where $i\colon Y\to X$ comes from a closed
immersion $T\to S$ by base change. We may assume $T=s$ is a closed
point of $S$. Then, replacing $S$ by its henselization $S_{(s)}$ at
$s$, with generic point $\eta$, we are reduced to showing that, for
any $K \in D^b_c(X_{\eta},G,E)$, the class in $\Kt (Y,G,E)$ of
\[
R\Gamma(I,R\Psi K) \in D^b_c(Y,G,E)
\]
is zero, where $I$ is the inertia subgroup of
$\Gal(\overline{\eta}/\eta)$, $\overline{\eta}$ an algebraic closure
of $\eta$. It then suffices to apply \ref{1.7}.

In the general case, apply the generic base change theorem
\cite[Th.\ finitude 1.9]{SGA4d} to find a dense open subset $V$ of
$S$ such that the formation of $j_* x$ commutes with any base change
$S'\to V$. Consider the following diagram with Cartesian squares
\[\xymatrix{U_s \ar[r]^{w_U}\ar[d]^{j_s} & U_V\ar[r]^{v_U}\ar[d]^{j_V} & U\ar[d]^{j} & \ar[l]_{t_U} U_T\ar[d]^{j_T}\\
X_s\ar[r]^{w}\ar[d] & X_V\ar[r]^{v}\ar[d] & X\ar[d] & X_T\ar[l]_t\ar[d]\\
s\ar[r] & V \ar[r] &S & \ar[l] T}\]
where $T=S-V$, $s$ is an arbitrary point of $V$. We have
$x=v_{U!}x_V + t_{U*} x_T$, where $x_V=v_U^* x$, $x_T=t_U^* x$.
Applying \ref{1.6} to $j_s$, we obtain
\[w^* j_{V*}(\xt_V) = j_{s*} w_U^*(\xt_V) = j_{s!} w_U^*(\xt_V) = w^* j_{V!}(\xt_V).\]
Hence, by \ref{2.3}, $j_{V*}(\xt_V)=j_{V!}(\xt_V)$. Therefore, by
the special case above,
\[j_* v_{U!} (\xt_V) = j_* v_{U*} (\xt_V) = v_*j_{V*}(\xt_V) = v_*j_{V!}(\xt_V) = v_!j_{V!}(\xt_V) = j_! v_{U!} (\xt_V).\]
On the other hand, applying \ref{1.6} to $j_T$, we obtain
\[j_* t_{U*} (\xt_T) = t_*j_{T*}(\xt_T) = t_*j_{T!}(\xt_T) = j_! t_{U*} (\xt_T).\]
Therefore $j_*(\xt )=j_!(\xt )$.
\end{ssect}

The following is a variant of \cite[6.2, 6.3]{LMB}. We use the convention of \cite[4.1]{LMB} for Artin stacks. In particular, the diagonal of an Artin stack is assumed to be quasi-compact and separated.

\begin{Lemma}\label{2.4}
Let $S$ be a quasi-separated scheme, $\cX$ be an Artin $S$-stack,
$K$ be an
$S$-field, $x=\Spec K$, $H$ be a finite group acting on $x$ on
the right by $S$-automorphisms, and let $i\colon [x/H] \to \cX$ be a morphism.
Then there exists a 2-commutative diagram
\[\xymatrix{ &[X/G]\ar[d]^\phi \\
[x/H] \ar[ur]^s\ar[r]^i & \cX}
\]
where $X$ is an affine scheme, $G$ is a finite group acting on $X$
on the right by $S$-automorphisms and $\phi$ is a representable smooth morphism. Moreover, if $\cX$ is a Deligne-Mumford $S$-stack, we can choose the above diagram such that $\phi$ is \'etale and that the following square is 2-Cartesian
\[\xymatrix{[x/H]\ar[r]^s\ar@{=}[d] & [X/G]\ar[d]^\phi\\
[x/H]\ar[r]^i & \cX}\]
\end{Lemma}

Let us recall the constructions in \cite[6.6]{LMB}. Let $\cX \to \cY$ be
a representable \cite[3.9]{LMB} and separated morphism of $S$-stacks,
$d\ge 0$. We define an $S$-stack $\SEC_d(\cX/\cY)$ by assigning to
every affine scheme $U$ equipped with a morphism $U\to S$, the
category of arrays $(x_1,\dots, x_d)$ of disjoint sections of the
algebraic $U$-space $X=\cX\times_\cY U$. The $S$-stack
$\SEC_d(\cX/\cY)$ is equipped with a natural action of the symmetric
group $\fS_d$, compatible with the projection to $\cY$. Let
$\ET_d(\cX/\cY)$ be the quotient stack. For a quasi-separated $S$-scheme
$V$, giving a morphism $V\to \ET_d(\cX/\cY)$ is equivalent to giving a morphism $V\to \cY$ and
giving a subscheme $Z$ of the algebraic space $\cX\times_\cY V$
which is finite, \'etale of degree $d$ over $V$ (\cite[6.6.3 (i)]{LMB}).
The structural morphism $\ET_d(\cX/\cY) \to \cY$ is representable
and separated \cite[6.6.3 (ii)]{LMB}.

\begin{proof}
We prove \ref{2.4} by imitating \cite[6.7]{LMB}. Take a 2-commutative diagram
with 2-Cartesian squares
\[\xymatrix{T\ar[r]\ar[d] & \cB \ar[r]\ar[d] & Z\ar[d]^{\pi}\\
x \ar[r] & [x/H] \ar[r]^{i} & \cX}
\]
with $Z$ an affine scheme, $\pi$ smooth and $T$ non-empty. Then $T$
is a smooth algebraic space over~$x$ and an $H$-torsor over $\cB$.
Let $L$ be an $H$-equivariant closed subscheme of~$T$, finite \'etale
over $x$. For example, we can take a closed point $t$ of $T$ whose residue field is a separable extension of $K$, and take $L$ to be the $H$-orbit of $t$, endowed with the
reduced algebraic subspace structure. Let $d$ be the degree of $L$ over $x$. We have thus
an $H$-equivariant section of $\ET_d(T/x) \to x$, giving rise to a section
of $\ET_d(\cB/[x/H]) \to [x/H]$, hence a 2-commutative diagram
\[\xymatrix{&&\ET_d(Z/\cX)\ar[d] \\
x\ar[rru]^{i_1}\ar[r] & [x/H]\ar[ru]\ar[r]^i & \cX}
\]
Since $\ET_d (Z/\cX)$ is a Deligne-Mumford stack smooth over $\cX$ \cite[6.6.3 (ii)]{LMB}, up to replacing $\cX$ by $\ET_d (Z/\cX)$, we may assume that $\cX$
is a Deligne-Mumford $S$-stack. In this case, we
can take $\pi$ to be \'etale. Then $\SEC_d(Z/\cX)$ is a quasi-affine
scheme \cite[6.6.2 (iii))]{LMB} and $\ET_d(Z/\cX) =
[\SEC_d(Z/\cX)/\fS_d]$ is \'etale over $\cX$. Moreover, since $T$ is a quasi-compact \'etale algebraic space over $x$, it is finite and we can take $L=T$, so that $\ET_d(\cB/[x/H])\to [x/H]$ is an isomorphism. The point $i_1$ corresponds to an
$\fS_d$-orbit of $\SEC_d(Z/\cX)$, which is contained in an
$\fS_d$-equivariant affine open. Thus \ref{2.4} holds by taking $X$ to be
the aforementioned open and $G$ to be the group $\fS_d$.
\end{proof}

\begin{ssect}\label{2.5} \begin{proof}[Proof of \ref{2.2}]
By \ref{2.3}, it is enough to show $i_\xi^* f_*
(\xt) = i_\xi^* f_! (\xt)$ for all points $\xi$ of $\cY$. By \ref{2.4}, it
is then enough to show $\phi^* f_* (\xt) = \phi^* f_! (\xt)$ for
every representable \'etale morphism $\phi \colon [Y/H] \to \cY$ where $Y$ is an affine
$S$-scheme of finite type and $G$ is a finite group acting on $Y$.
By base change by $\phi$, it is thus enough to establish \ref{2.2} in the
case where $\cY = [Y/H]$. In particular, \ref{2.4.1} implies that \ref{2.2} holds
if $f$ is an open immersion (with no additional assumption on $\cY$).

We prove \ref{2.2} in the case where $\cY = [Y/H]$ by Noetherian
induction on $\cX$. Let $j\colon [X/G] \to \cX$ be a dominant open
immersion \cite[6.1.1]{LMB}, where $X$ is an affine $S$-scheme of
finite type and $G$ is a finite group acting on $X$. By
\cite[5.1]{Z} (which holds over general base schemes), up to
replacing $(X,G)$ by another pair with an isomorphic quotient stack,
we may assume that $fj \colon [X/G] \to [Y/H]$ is induced by an
equivariant map $(X,G) \to (Y,H)$. Let $i$ be a closed immersion
$\cX - [X/G] \to \cX$. Then $x= j_! j^* x + i_* i^* x$. Since $j_*
j^* (\xt) = j_! j^* (\xt)$ by the already proven case of open
immersion, we have
\[ f_* j_!j^* (\xt) = f_* j_* j^* (\xt) = (fj)_* j^* (\xt) = (fj)_!
j^* (\xt) = f_! j_! j^* (\xt) \]
by \ref{2.4.1} applied to $fj$. On the other hand,
\[ f_* i_* i^* (\xt) = (fi)_* i^* (\xt) = (fi)_! i^* (\xt) = f_! i_*
i^* (\xt) \]
by induction hypothesis applied to $fi$. Therefore $f_*(\xt) =
f_!(\xt)$.
\end{proof}\end{ssect}

In the rest of this section, we consider analogues of Laumon's theorem for torsion coefficients. Let $F$ be a field of characteristic $\ell$. The following is an analogue of Theorem \ref{1.7}.

\begin{Theorem}\label{2.6}
Let $S$ be the spectrum of a henselian discrete valuation ring, with closed point~$s$ and generic point $\eta$, $Y$ be a scheme of finite type over $s$, endowed with an action of a finite group $G$ ($G$ acting trivially on~$s$). Then for any $L\in \Dbc(Y\times_s \eta, G, F)$, the class in $\Kt (Y,G,F)$ of \[R\Gamma(I,L)\in \Dbc(Y,G,F)\]
is zero, where $I$ is the inertia subgroup of $\Gal(\overline{\eta}/\eta)$, $\overline{\eta}$ an algebraic closure of
$\eta$.
\end{Theorem}

As in \ref{1.7}, one is reduced to showing that, for any
$G$-${\Gal}(\eta_\ell/\eta)$-$F$-sheaf, the class of
$R\Gamma(I_{\ell},L) $ in $\Kt (Y,G,F)$ is zero, where $\eta_{\ell}
= \overline{\eta}/P_{\ell}$. As before, we may assume that the
action of $I_{\ell}$ on $L$ is unipotent. Fix a topological
generator $\sigma$ of $I_\ell = \Z_\ell(1)$ and define a (nilpotent)
$G$-$I_\ell$-equivariant operator
\[N_\sigma\colon L(1)\to L\]
by the formula $N_\sigma(\bar \sigma \otimes a) =u a$, where
$\bar \sigma\in F(1)=F\otimes_{\Z_\ell}\Z_\ell(1)$ is the image of $\sigma$, $u=\sigma-1\colon L\to L$. For $\gamma \in
\Gal(\eta_\ell/\eta)$, we have
\[(N_\sigma\gamma -\gamma N_\sigma)(\bar \sigma \otimes a) = N_\sigma((\bar\sigma)^{\chi(\gamma)}
\otimes \gamma a) - \gamma (\sigma-1)a =
\overline{\chi(\gamma)}u\gamma a - (\sigma^{\chi(\gamma)}-1)\gamma
a, \] where $\chi\colon \Gal(\eta_\ell/\eta)\to \Z^\times_\ell$ is
the cyclotomic character, $\overline{\chi(\gamma)}\in
\F_\ell^\times$ is the image of $\chi(\gamma)\in \Z_\ell^\times$.
Since
\[\sigma^{\chi(\gamma)}-1 = (1+u)^{\chi(\gamma)}-1 = \overline{\chi(\gamma)}u+u^2 P(u),\]
where $P(u)$ is a polynomial in $u$, we have
\[(N_\sigma\gamma -\gamma N_\sigma)(\bar \sigma \otimes a)
=\overline{\chi(\gamma)} u \gamma a - [\overline{\chi(\gamma)} u
\gamma a + u^2 P(u) \gamma a] = -u^2 P(u) \gamma a \in \Img u^2 =
\Img N_\sigma^2.\] It follows that $\Img (N_\sigma^m \gamma - \gamma
N_\sigma^m)\subset \Img(N_\sigma^{m+1})$ for $m\ge 0$. Let $\dots
\subset M_i L \subset M_{i+1} L \subset \dotsb$ be the filtration in
the category of $G$-$F$-sheaves on $Y_{\overline{s}}$ characterized
by $N_\sigma M_i(1) \subset M_{i-2}$ and the property that
$N_\sigma^i$ induces an isomorphism of $G$-$F$-sheaves
\begin{equation}
\Gr^M_i L(i)
\xrightarrow{\sim} \Gr^M_{-i} L. \label{(2.6.1)}
\end{equation}
As $\dots
\subset \gamma M_i L \subset \gamma M_{i+1} L \subset \dotsb$
satisfies the same condition for any $\gamma \in
\Gal(\eta_\ell/\eta)$, the filtration is
$\Gal(\eta_\ell/\eta)$-stable. Moreover, \eqref{(2.6.1)} is
$\Gal(\eta_\ell/\eta)$-equivariant, hence the same holds for the
isomorphism
\[\Gr^M_i(\Coker N_\sigma)(i) \simeq \Gr^M_{-i}(\Ker N_\sigma (-1)).\]
Therefore $H^0(I_\ell,L) = L^{I_\ell}=\Ker N_\sigma(-1)$ and
$H^1(I_\ell,L) = L_{I_\ell} (-1)=\Coker N_\sigma (-1)$ have the same class in
$\Kt(Y,G,F)$.

\begin{sRemark}
Note that the filtration $M_i L$ in the proof does not depend on the
choice of $\sigma$. In fact, if $\tau=\sigma^r$ is another topological generator of $\Z_\ell(1)$, $r\in \Z_\ell^\times$, then
\[\bar r N_\tau (\bar \sigma \otimes a) = N_\tau(\bar \tau \otimes a) = (\tau -1) a = [(1+u)^r-1]a = \bar r ua + u^2 Q(u) a,\]
where $Q(u)$ is a polynomial in $u=\sigma-1$, hence
\[(N_\tau - N_\sigma)(\bar \sigma \otimes a)= (\bar r)^{-1} u^2 Q(u) a \in \Img N_\sigma^2.\]
\end{sRemark}

To state an analogue of \ref{2.2} for $F$-sheaves, we need a condition on inertia.

\begin{Proposition}\label{2.7}
Let $S$ be a quasi-separated scheme, $f\colon \cX\to \cY$ be a morphism of Deligne-Mumford $S$-stacks, $m\in \Z$ be an integer. The following two conditions are equivalent:
\begin{enumerate}
\item For any algebraically closed field $\Omega$ and any point $x\in
\cX(\Omega)$, the order of the group $\Ker(\Aut_{\cX}(x)\to \Aut_{\cY}(y))$ is prime to $m$, where $y\in \cY(\Omega)$ is the image of $x$ under~$f$;

\item For any algebraically closed field $\Omega$ and any point $y\in \cY(\Omega)$ and any lifting $x\in \cX(\Omega)$ of $y$, the order of the group $\Aut_{\cX_y}(x)$ is prime to $m$.
\end{enumerate}
\end{Proposition}

Note that the morphisms satisfying (a) are closed under composition while the morphisms satisfying (b) are closed under base change.

\begin{Definition}\label{d.inert}
Morphisms satisfying the conditions of \ref{2.7} are called \emph{of prime to $m$ inertia}.
\end{Definition}

\begin{Example}
Let $(X,G)\to (Y,H)$ be an equivariant morphism of $S$-schemes (the finite groups $G$ and $H$ acting trivially on $S$). If for all geometric points $x\to X$, of image $y\to Y$, the order of the group $\Ker(G_x\to H_y)$ is prime to $m$, where $G_x$ and $H_y$ are the inertia groups, then the induced morphism $[X/G]\to [Y/H]$ of quotient $S$-stacks is of prime to $m$ inertia. Indeed, $\Aut_{[X/G]}(\xi)=G_x$, $\Aut_{[Y/H]}(\upsilon)=H_y$, where $\xi$ is the composition $x\to X\to [X/G]$, $\upsilon$ is the composition $y\to Y\to [Y/H]$.
\end{Example}

We will deduce \ref{2.7} in \ref{s.inert} from some general facts
about inertia. Let $\cC$ be a 2-category. The 2-commutative squares
in $\cC$ form a 2-category $\cC^{\square}$ in an obvious way. Let
$S$ be the partially ordered set
\[\{0,1\}^2=\{(i,j)\mid 0\le i, j \le 1\}.\]
A pseudofunctor $F\colon S\to \cC$
is a 2-commutative diagram in $\cC$ of the form
\begin{equation}\label{e.square}
\xymatrix{F_{00}\ar[r]\ar[d]\ar[rd] & F_{01}\ar[d]\\
F_{10}\ar[r] & F_{11}}
\end{equation}
The 2-category $\cC^\square$ can be identified with the
2-subcategory of the 2-category of pseudofunctors $S\to \cC$,
spanned by those pseudofunctors $F$ for which the lower-left
triangle is strictly commutative. A \emph{2-Cartesian square} is a 2-commutative square of the form \eqref{e.square} which is a 2-limit diagram, namely which exhibits $F_{00}$ as the 2-limit of the diagram indexed by $T=( (1,0)\to (1,1)\gets (0,1))\subset S$.

\begin{Lemma}\label{2.8}
Consider a 2-commutative square
\[\xymatrix{A\ar[r]\ar[d]&B\ar[d]\\
X\ar[r]&Y}
\]
in $\cC^\square$. Suppose that the squares $A_{ij}B_{ij}X_{ij}Y_{ij}$, $0\le i,j
\le 1$ and the squares $B$, $X$, $Y$ are 2-Cartesian. Then the
square $A$ is 2-Cartesian.
\end{Lemma}

\begin{proof}
The restriction $X|T\to Y|T \gets B|T$ of the given square corresponds to a diagram in $\cC$ indexed by $T\times T'$, which we denote by $D$. Here $T'=T$. Then $A_{00}$ is $\lim_{T'}\lim_{T} D$, where $\lim$ stands for 2-limit. The assertion then follows from the canonical identification of $\lim_T\lim_{T'}D$ and $\lim_{T'}\lim_T D$.
\end{proof}

\begin{Corollary}\label{p.cube}
  Let
  \[\xymatrix{&A'\ar[rr] \ar[ld]\ar[dd]|\hole && A\ar[dd]\ar[ld]\\
  B'\ar[rr]\ar[dd]&& B\ar[dd]\\
  &X'\ar[rr]|\hole\ar[ld]&&X\ar[ld]\\
  Y'\ar[rr]&&Y}\]
  be a 2-commutative cube in a 2-category. If the front, back and bottom squares are 2-Cartesian, then the top square is also 2-Cartesian.
\end{Corollary}

\begin{proof}
  It suffices to apply \ref{2.8} to the square of squares
  \[\xymatrix{A'AB'B\ar[r]\ar[d]& AABB\ar[d]\\
  X'XY'Y\ar[r]& XXYY}\]
\end{proof}

\begin{ssect}
Let $\cC$ be a 2-category. We say that a morphism $f\colon X\to Y$ in $\cC$ is \emph{faithful} if for every object $W$ of $\cC$, the functor
\[\Hom_{\cC}(W,X)\to \Hom_{\cC}(W,Y)\]
is faithful. Assume $\cC$ admits 2-fiber products. For any morphism $f\colon X\to Y$ in $\cC$, we define the \emph{inertia} of $f$ to be
\begin{equation}
I_f =X\times_{\Delta_{f},X\times_Y X, \Delta_{f}}X.
\end{equation}
The two projections $X\times_Y X\to X$ induce two isomorphisms
between the two projections $I_f\to X$. To fix ideas, we endow $I_f$
with the first projection to $X$, which is a faithful morphism
because it admits the diagonal morphism $\delta_f\colon X\to I_f$ as a section. For a morphism
$g\colon W\to X$ in $\cC$, $\Hom_X(W,I_f)$ is equivalent to the
group $\Aut_{\cD}(g)$, with the diagonal section $\delta_f g$
corresponding to the identity element of the group. Here $\cD$ is
the category $\Hom_Y(W,X)$. In the case where $\cC$ is the
2-category of categories fibered over a given category
$\mathcal{A}$, an explicit description of $I_f$ can be found in
\cite[034H]{Stacks}.
\end{ssect}

\begin{Lemma}\label{p.ICart}
Let
\[\xymatrix{X'\ar[r]\ar[d] & X\ar[d]\\
Y'\ar[r] & Y}\]
be a 2-Cartesian square in a 2-category admitting 2-fiber products. Then the square
\[\xymatrix{I_{X'/Y'} \ar[r]\ar[d] & I_{X/Y}\ar[d]\\
X'\ar[r]& X}\]
is 2-Cartesian.
\end{Lemma}

\begin{proof}
  It suffices to apply \ref{p.cube} successively to the following cubes:
  
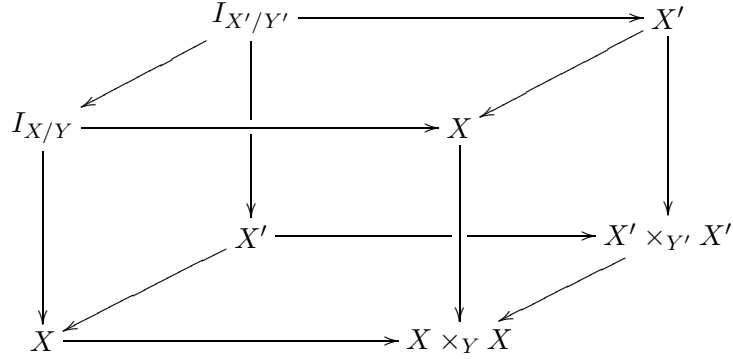
\[\xymatrix@!C{&X'\times_{Y'}X'\ar[rr] \ar[ld]\ar[dd]|\hole && X'\ar[dd]\ar[ld]\\
  X\times_{Y}X\ar[rr]\ar[dd]&& X\ar[dd]\\
  &X'\ar[rr]|\hole\ar[ld]&&Y'\ar[ld]\\
  X\ar[rr]&&Y}\]
  \[\xymatrix{&X'\ar[rr] \ar[ld]\ar@{=}[dd]|\hole && X\ar@{=}[dd]\ar[ld]\\
  X'\times_{Y'}X'\ar[rr]\ar[dd]&& X\times_Y X\ar[dd]\\
  &X'\ar[rr]|\hole\ar[ld]&&X\ar@{=}[ld]\\
  X'\ar[rr]&&X}\]
  \[\xymatrix@!C{&I_{X'/Y'}\ar[rr] \ar[ld]\ar[dd]|\hole && X'\ar[dd]\ar[ld]\\
  I_{X/Y}\ar[rr]\ar[dd]&& X\ar[dd]\\
  &X'\ar[rr]|\hole\ar[ld]&&X'\times_{Y'}X'\ar[ld]\\
  X\ar[rr]&&X\times_Y X}\]
\end{proof}

\begin{Lemma}\label{p.Icomp}
  Let $X\to Y\to Z$ be a sequence of morphisms in a 2-category admitting 2-fiber products. Then the canonical morphism $I_{X/Y}\to I_{X/Z}$ induces an isomorphism
  \[I_{X/Y} \longsimto K=\Ker(I_{X/Z} \to X\times_{Y} I_{Y/Z}). \]
\end{Lemma}

Here $K$ is defined by the following 2-commutative diagram with
2-Cartesian squares
  \[\xymatrix{K \ar[r]\ar[d] & X\ar[r]\ar[d] & Y\ar[d]^{\delta_{Y/Z}}\\
  I_{X/Z} \ar[r] & X \times_Y I_{Y/Z} \ar[r] & I_{Y/Z}}\]

\begin{proof}
  Applying \ref{2.8} successively to the squares of squares
  \[\xymatrix{X\times_Y X, Y, X\times_{Z}X, Y\times_Z Y\ar[r]\ar[d] & XYXY\ar[d]\\
  XYXY\ar[r] & YYZZ}\]
  and
  \[\xymatrix{I_{X/Y}, Y, I_{X/Z}, I_{Y/Z} \ar[r]\ar[d] & XYXY\ar[d]\\
  XYXY\ar[r] & X\times_Y X, Y, X\times_{Z}X, Y\times_Z Y}\]
  we see that the outer square of the 2-commutative diagram
  \[\xymatrix{I_{X/Y} \ar[r]\ar[d] & X\ar[r]\ar[d] & Y\ar[d]\\
  I_{X/Z} \ar[r] & X \times_Y I_{Y/Z} \ar[r] & I_{Y/Z}}\]
  is 2-Cartesian. Since the square on the right of the above diagram is 2-Cartesian, it follows from \ref{p.cube} that the square on the left is also 2-Cartesian.
\end{proof}


\begin{ssect}\label{s.inert}
\begin{proof}[Proof of \ref{2.7}]
By \ref{p.Icomp},
\[I_{\cX/\cY}\simeq \Ker (I_{\cX/S} \to \cX\times_\cY I_{\cY/S}).\]
By \ref{p.ICart},
\[I_{\cX_y/y}\simeq y\times_{\cY} I_{\cX/\cY}\]
for $y\in \cY(\Omega)$.
Therefore,
\[\Aut_{\cX_y}(x) \simeq \Hom_{\cX_y}(x,I_{\cX_{y}/y})\simeq \Hom_\cX(x,I_{\cX/\cY})\simeq \Ker(\Aut_\cX(x) \to \Aut_{\cY}(y))\]
for $x\in \cX(\Omega)$ lifting $y$.
\end{proof}
\end{ssect}

\begin{ssect}\label{d.fin}
Now let $S$ be as in \ref{2.1}, $f\colon \cX\to \cY$ be a morphism of finite type Deligne-Mumford
$S$-stacks of prime to $\ell$ inertia. Recall that $F$ is a field of characteristic~$\ell$. Then we have functors \cite[\S~2]{Zhengl} 
\[Rf_*, Rf_! \colon \Dbc(\cX,F) \to \Dbc(\cY,F).\]
They induce homomorphisms
\[f_*, f_! \colon K(\cX,F) \to K(\cY, F) \]
and
\[f_*, f_! \colon \Kt(\cX,F) \to \Kt(\cY, F) \]
by passing to quotients.
\end{ssect}

\begin{Theorem}\label{2.10}
Let $f\colon \cX\to \cY$ be a morphism of finite type Deligne-Mumford
$S$-stacks of prime to $\ell$ inertia.
For any $x\in K(\cX,F)$, $f_*(\xt) = f_!(\xt) $ in
$\Kt(\cY,F)$.
\end{Theorem}

The proof follows the same line as the proof of \ref{2.2}. In
\ref{2.5}, up to shrinking $X$, we may assume in addition that
$G/G_0$ acts freely on $X$, where $G_0=\Ker (G\to \Aut(X))$. As
before, by \cite[5.1]{Z}, we are reduced to showing \ref{2.10} in
the case where $f\colon [X/G]\to [Y/H]$ is induced by an equivariant
morphism $(X,G)\to (Y,H)$ of affine schemes of finite type over $S$,
with $G/G_0$ acting freely on $X$. Then $f$ is the composite
morphism
\[[X/G] \to [(X/G_0\cap N)/(G/G_0\cap N)]\xrightarrow{g} [(X/N)/(G/N)] \to [Y/H],\]
where $N=\Ker (G\to H)$. Since $N/G_0\cap N$ acts freely on $X$, $g$
is an isomorphism. By assumption, $G_0\cap N$ has order prime to
$\ell$, hence we are reduced to showing \ref{2.10} in the case where
$f\colon [X/G]\to [Y/H]$ is induced by an equivariant morphism
$(X,G)\to (Y,H)$ of affine schemes of finite type over $S$, with
$\ell$ prime to the order of $\Ker (G\to H)$. We prove the analogues
of \ref{1.6}, \ref{1.2} and \ref{2.4.1} as before, with \ref{1.7}
replaced by \ref{2.6}. (For the analogues of \ref{1.2} and
\ref{2.4.1}, we assume that $\Ker(G\to H)$ has order prime to
$\ell$.)

\section{Free actions and vanishing theorems: $\ell$-adic and
Betti cohomologies}
\begin{ssect}\label{s.space}
Let $k$ be an algebraically closed field of characteristic exponent
$p$, $\ell$ a prime number $\ne p$, $X$ an algebraic $k$-space separated and of
finite type, endowed with an action of a
finite group~$G$. We extend the notations $t_\ell(s)$, $t_{c,\ell}(s)$ and $\chi(X,G,\Q_\ell)$ (\eqref{(0.1)} through \eqref{(0.3)}) to this situation. By \ref{2.2}, $t_\ell(s)=t_{c,\ell}(s)$ for $s\in G$.
\end{ssect}

\begin{Theorem}\label{3.1}
Under the assumptions of \ref{s.space}, we have:
\begin{enumerate}
\item For $s \in G$, $t_{\ell}(s)$ is an integer $t(s)$
    independent of $\ell$;

\item If $G$ acts freely on $X$,
    $R\Gamma_c(X,{\mathbb{Z}}_{\ell})$ (resp.
    $R\Gamma(X,{\mathbb{Z}}_{\ell})$) is a perfect complex of
    ${\mathbb{Z}}_{\ell}[G]$-modules (\ie is isomorphic to a
    bounded complex of finitely generated projective
    ${\mathbb{Z}}_{\ell}[G]$-modules);

\item If $G$ acts freely on $X$, then $t(s) = 0$ for every $s \in
    G$ whose order is not a power of~$p$.
\end{enumerate}
\end{Theorem}

\begin{proof}
Assertion (a) for $t_{c,\ell}(s)$ in the case of  a separated scheme
is a result of Deligne-Lusztig \cite[3.3]{DL}. The general case is
similar: by additivity of $t_{c,\ell}$ with respect to $X$, we may
assume $X$ affine; by spreading out one reduces to the case where $k$
is the algebraic closure of a finite field ${\mathbb{F}}_q$ and $X$
with its action of $G$ is defined over ${\mathbb{F}}_q$; in this
case, if $F$ is the geometric Frobenius of $\Spec
k/\Spec{\mathbb{F}}_q$, for any $n \ge 1$, $s \times_{{\mathbb{F}}_q}
F^n$ is identified with $\Id_{X'_n} \times_{{\mathbb{F}}_{q^n}} F^n$
for a suitable $X'_n/{\mathbb{F}}_{q^n}$, and the assertion follows
from Grothendieck's trace formula.

The argument for (b) and (c) is analogous to that of \cite[proof of
2.5]{I1}. We have \beq R\Gamma(X,{\mathbb{Z}}_{\ell}) =
R\Gamma(X/G,\pi_*{\mathbb{Z}}_{\ell}), \label{(3.1.1)} \eeq where
$\pi \colon X \rightarrow X/G$ is the projection. Here $X/G$ is an
algebraic space of finite type over~$k$. As $\pi$ is an \'{e}tale
Galois cover of group $G$, $\pi_*{\mathbb{Z}}_{\ell}$ is a lisse
sheaf, locally free of rank one over ${\mathbb{Z}}_{\ell}[G]$. For
any ${\mathbb{Z}}_{\ell}[G]$-module $M$ of finite type, we have, by
the projection formula, \beq\label{(3.1.2)}
R\Gamma(X/G,\pi_*{\mathbb{Z}}_{\ell})
\otimes^L_{{\mathbb{Z}}_{\ell}[G]} M \longsimto
R\Gamma(X/G,\pi_*{\mathbb{Z}}_{\ell} \otimes_{{\mathbb{Z}}_{\ell}[G]}
M), \eeq which implies that $R\Gamma(X/G,\pi_*{\mathbb{Z}}_{\ell})$
is of finite tor-dimension,
hence perfect (as $R\Gamma(X/G,\pi_*{\mathbb{Z}}_{\ell})$ belongs to $D^b_c({\mathbb{Z}}_{\ell})$ by \eqref{(3.1.1)}
). (This type of argument
seems to have appeared for the first time in Grothendieck's proof of the
Euler-Poincar\'{e} and Lefschetz formulas for curves,
cf. \cite[III, X]{SGA5}, \cite[Rapport]{SGA4d}.)
The proof for $R\Gamma_c$ is analogous, except that we may assume $X/G$ to be a separated scheme by induction on $\dim X$.

By (a) the character $t$ of $\chi_c(X,G,{\mathbb{Q}}_{\ell})$ (=
$\chi(X,G,{\mathbb{Q}}_{\ell})$ \eqref{(1.3.2)}) has values in
${\mathbb{Z}}$ and is independent of $\ell$. Therefore (c) follows
from the theory of modular characters \cite[th.~36, p.~145]{Se1}: if
$P$ is a finitely generated projective
${\mathbb{Z}}_{\ell}[G]$-module, the character of $P \otimes
{\mathbb{Q}}_{\ell}$ vanishes on $\ell$-singular elements of $G$, \ie
elements whose order is divisible by~$\ell$.
\end{proof}

For $s = 1$, the fact that $t_c(1) = \chi_c(X,{\mathbb{Q}}_{\ell})$
is independent of $\ell$ had been known since the early 1960s, as it
is an immediate consequence of Grothendieck's trace formula (cf.
\cite[\S~1]{I2}). Thanks to Gabber's theorem \cite{Fu}, one can show
the independence of $\ell$ for $t_\ell(s)$ independently of
\ref{1.3}.

As observed in \cite[3.12]{DL}, \ref{3.1} implies:

\begin{Corollary}\label{3.2}
If $G$ acts freely on $X$ and the order of $G$
is prime to $p$, then, with the notations of \ref{1.3},
\beq
\chi(X,G,{\mathbb{Q}}_{\ell}) = \chi(X/G)\Reg_{{\mathbb{Q}}_{\ell}}(G), \label{(3.2.1)}
\eeq
where $\Reg_{{\mathbb{Q}}_{\ell}}(G)$ is the class of the
regular representation and
\[
\chi(X/G) = \sum (-1)^i\dim H^i_c(X/G,{\mathbb{Q}}_{\ell}) =
\sum (-1)^i\dim H^i(X/G,{\mathbb{Q}}_{\ell})
\]
is the
Euler-Poincar\'{e} characteristic of the algebraic space $X/G$.
\end{Corollary}

\begin{proof}
Indeed,
by \ref{3.1}, we have $t(s) = 0$ for all $s \ne 1$. Therefore, by
\cite[2.4, 12.1]{Se1}, there exists $m \in {\mathbb{Z}}$ such that
\[
\chi(X,G,{\mathbb{Q}}_{\ell}) = m\Reg_{{\mathbb{Q}}_{\ell}}(G).
\]
As, by \eqref{(3.1.2)},
\[
R\Gamma(X/G,\pi_*{\mathbb{Z}}_{\ell}) \otimes^L_{{\mathbb{Z}}_{\ell}[G]} {\mathbb{Z}}_{\ell} =
R\Gamma(X/G,{\mathbb{Z}}_{\ell}),
\]
one finds $m = \chi(X/G,{\mathbb{Q}}_{\ell})$.
\end{proof}

The following application was suggested to the first author by Serre:

\begin{Corollary}\label{3.3}
With the notations of \ref{s.space}, assume that $G$ is cyclic, generated by
$s$.  Assume moreover that the order of $G$ is prime to $p$ and that
$s$ has no fixed
points. Then $t(s) = 0$.
\end{Corollary}

\begin{proof}
When $G$ acts freely, this is a particular case of \ref{3.2}. In the
general case, for any subgroup $H$ of $G$, denote by $X^H$ the fixed
point set of $G$ (a closed algebraic subset of $X$) and, as in \cite[\S~2]{V}, by $X_H$ the open
subset of $X^H$ defined by
\[
X_H = X^H - \bigcup_{H' \supset H, H' \ne H} X^{H'}.
\]
Each $X^H$ (resp. $X_H$) is $G$-stable, and the inertia group at any
point of $X_H$ is $H$. On $X_H$ the quotient $G/H$ acts freely. As
$X$ is the disjoint union of the $X_H$'s for $H$ running through the
subgroups of $G$, we have
\[
t(s) = \sum_{H} \Tr(s,H^*_c(X_H,{\mathbb{Q}}_{\ell})).
\]
As $s$ generates $G/H$, $\Tr(s,H^*_c(X_H,{\mathbb{Q}}_{\ell})) = 0$, hence $t(s) = 0$.
\end{proof}

Finally, here is an application to Betti cohomology:

\begin{Corollary}\label{3.4}
Let $X/{\mathbb{C}}$ be an algebraic space separated and of finite type over ${\mathbb{C}}$
endowed with a free action of a finite group $G$. Let
$R_{{\mathbb{Q}}}(G)$ denote the Grothendieck group of
${\mathbb{Q}}$-linear finite dimensional representations of~$G$, and
\begin{align*}
\chi_c(X,G,{\mathbb{Q}}) &= \sum (-1)^i[H^i_c(X,{\mathbb{Q}})] \in R_{{\mathbb{Q}}}(G),
\\
\chi(X,G,{\mathbb{Q}}) &= \sum (-1)^i[H^i(X,{\mathbb{Q}})] \in R_{{\mathbb{Q}}}(G),
\end{align*}
where $[-]$ denotes a class in $R_{{\mathbb{Q}}}(G)$. Then:
\beq
\chi_c(X,G,{\mathbb{Q}}) = \chi(X,G,{\mathbb{Q}}) =\chi(X/G,\Q)\Reg_{\mathbb{Q}(G)},
\label{(3.4.1)}
\eeq
where $\Reg_{{\mathbb{Q}}}(G)$ is the class of the regular
representation and
\[
\chi_c(X/G,{\mathbb{Q}}) = \chi(X/G,{\mathbb{Q}})
\]
is the Euler-Poincar\'{e}
characteristic of $X/G$. (Here by $H^i_c(Y,\mathbb{Q})$ (resp.\ $H^i(Y,\mathbb{Q})$), for
$Y/{\mathbb{C}}$ an algebraic space separated and of finite type, we mean $H^i_c(Y({\mathbb{C}}),\mathbb{Q})$ (resp.\ $H^i(Y({\mathbb{C}}), \mathbb{Q})$), where
$Y({\mathbb{C}})$ is the space of
rational points of $Y$ with the classical topology defined in \cite[1.6]{Artin}.)
\end{Corollary}

Recall the comparison theorem between
\'{e}tale and Betti cohomologies for algebraic spaces: for a morphism $f\colon Y\to Z$ of finite type algebraic spaces over $\C$, we have a 2-commutative square of topoi
\[\xymatrix{Y(\C)_{\cl}\ar[d]\ar[r] & Y_{\et}\ar[d]\\
Z(\C)_{\cl}\ar[r] & Z_{\et}}\]
and the base change morphism $Rf_{\cl *} \epsilon_Z^* {\cal F} \to \epsilon_Y^* Rf_{{\et}*} {\cal F}$ is an isomorphism for any constructible torsion abelian sheaf $\cal F$. See \cite[XVI 4.1]{SGA4} for the case of schemes. The general case follows from this case: the problem being local, we may assume $Z$ to be a scheme; then we take an \'etale cover of $Y$ by schemes and apply cohomological descent. In particular, $\chi(Y,{\mathbb{Q}})=\chi(Y,\Q_\ell)$
(see \ref{d.fin} for finiteness).
We also have $\chi_c(Y,{\mathbb{Q}})=\chi_c(Y,\Q_\ell)$ for $Y$ separated. Indeed, it suffices to take a
stratification by separated schemes and apply the identity to each stratum \cite[Arcata IV.6.3]{SGA4d}.

The homomorphism $R_{{\mathbb{Q}}}(G) \rightarrow R_{{\mathbb{Q}}_{\ell}}(G)$
is injective \cite[14.6]{Se1}, and by the comparison theorem, $\chi_c(X,G,{\mathbb{Q}}_{\ell})$ (resp.\ $\chi(X,G,{\mathbb{Q}}_{\ell})$) is the image of $\chi_c(X,G,{\mathbb{Q}})$ (resp.\ $\chi(X,G,{\mathbb{Q}})$).
Therefore \eqref{(3.4.1)} follows from \eqref{(3.2.1)}.

\begin{Remark}\label{3.5}
(a) The equality
\beq
\chi_c(X,G,{\mathbb{Q}}) = \chi_c(X/G)\Reg_{\mathbb{Q}}(G) \label{(3.5.1)}
\eeq was established by Verdier
\cite[lemme, p.~443]{V}, using a similar method. At the time,
the equality $\chi = \chi_c$ was unknown.
Actually, Verdier proves \eqref{(3.4.1)} more generally for locally compact spaces $X$
which are of \emph{finite topological dimension} (\ie such that
there exists an integer $N$ such that $H^n_c(X,F) = 0$ for all $n> N$
and all abelian sheaves~$F$) and \emph{cohomologically of finite type}
(\ie such that, for
each $n$, $H^n_c(X,{\mathbb{Z}})$ is finitely generated), endowed with a
continuous and free action of a finite group $G$ ($X/G$ is then also
cohomologically of finite type). When the action of $G$ is no longer
assumed to be free, but the fixed point set $X^H$ is cohomologically
of finite type for every subgroup $H$ of $G$, using a decomposition of
$X$ of the type considered in the proof of \ref{3.3}, Verdier formally deduces from
\eqref{(3.4.1)} a general formula
for $\chi_c(X,G,\Q)$ as a linear combination of certain induced
representations (\loccit, p.~443). We will come back to this
and $\ell$-adic variants in \S~4.

(b) By a theorem of
Deligne reported on in \cite{I1}, if $X$ is a separated scheme, \eqref{(3.2.1)} holds more generally if the action of
$G$ is only assumed to be \emph{tame at infinity}. We will discuss this
and related results in \S~4.
\end{Remark}

\begin{ssect}\label{3.6} Let $k$ be a field, $\overline{k}$ a separable closure of $k$,
$X$ an algebraic $k$-space of finite type, endowed with a free action
of a group $G$ of order $\ell^n$, where $\ell$ is a prime different
from the characteristic of $k$. Serre proved (in the case of a
separated scheme) that, under these assumptions, for any $g \in
\Gamma = \Gal(\overline{k}/k)$, the $\ell$-adic integer
$\Tr(g,H^*_c(X_{\overline{k}},{\mathbb{Q}}_{\ell}))$ is divisible by
$\ell^n$ (\cite{Se2}, \cite[7.5]{I2}). On the other hand, by
\ref{3.1} (c), for $s \in G$, $s \ne 1$,
$\Tr(s,H^*_c(X_{\overline{k}},{\mathbb{Q}}_{\ell})) = 0$.  More
generally:
\end{ssect}

\begin{Theorem}\label{3.7}
Under the assumptions of \ref{3.6}, for any $s \in G$ and any $g \in
\Gamma$, we have
\beq
\Tr(sg,H^*_c(X_{\overline{k}},{\mathbb{Q}}_{\ell})) \equiv 0 \mod{\ell^m}, \label{(3.7.1)}
\eeq
where $\ell^m$ is the order of the centralizer of $s$ in $G$.
\end{Theorem}

The proof follows the pattern of that of Serre's theorem (\loccit).
We may assume $X$ affine. Choose a model ${\cal X}/S$ of $X/k$ where
$S$ is the spectrum of a finitely generated, integral, normal, sub
${\mathbb{Z}}$-algebra of $k$,  ${\cal X}/S$ affine and of finite
type, endowed with a free action of $G$ by $S$-automorphisms, such
that $X/k$, with its $G$-action, comes from ${\cal X}/S$ by base
change. Suppose that there exists a pair $(s,g)$ such that $t_c(sg)
\not \equiv 0 \mod{\ell^m}$, where $t_c(sg) =
\Tr(sg,H^*_c(X_{\overline{k}},{\mathbb{Q}}_{\ell}))$. Applying
Chebotarev's generalized density theorem as in \cite[proof of 7.1,
(1) $\Rightarrow$ (2)]{I2}, one finds a point $y$ of $S$ with value
in a finite field ${\mathbb{F}}_q$ of characteristic $p \ne \ell$
such that $t_c(sg) \equiv \Tr(sF,H^*_c({\cal
X}_{\overline{y}},{\mathbb{Q}}_{\ell})) \mod{\ell^m}$, where
$\overline{y}$ is an algebraic geometric point over $y$ and $F$ the
geometric Frobenius automorphism of $\overline{y}/y$. By
Deligne-Lusztig (\loccit), $sF$ is the geometric Frobenius
automorphism of $X'_{\overline{y}}$ for some $X'/{\mathbb{F}}_q$,
therefore, by Grothendieck's trace formula, the trace of $sF$ is the
cardinality of the set $E$ of rational points $x$ of ${\cal
X}_{\overline{y}}$ such that $sFx = x$. Since the action of $G$
commutes with $F$, the centralizer of $s$ in $G$ acts freely on $E$,
hence this cardinality is divisible by $\ell^m$, which is a
contradiction.

\begin{Remark}\label{3.8} (a) For $s \ne 1$, one can't expect $\Tr(sg,H^*_c(X_{\overline{k}},{\mathbb{Q}}_{\ell})) = 0$ for all $g \in
\Gamma$. Serre gives the following example: let $k = {\mathbb{F}}_q$
with $q-1 \equiv 0  \mod{\ell^n}$, and let $X = ({\mathbb{G}}_m)_k$; let $s$ be the translation in $X$ by a rational point of
order $\ell^n$ (a primitive $\ell^n$-th root of 1 in~$k$). Then the
trace of $sF$ is  equal to the trace of $F$, \ie $q-1$.

(b) We don't know whether \ref{3.7} holds with $H^*_c$ replaced by $H^*$.
\end{Remark}

\begin{ssect}\label{3.9} Under the assumptions of \ref{s.space}, when $G$ is not assumed to act
freely, $R\Gamma(X,{\mathbb{Z}}_{\ell})$ (resp.
$R\Gamma_c(X,{\mathbb{Z}}_{\ell})$), as an object of $D^b_c({\mathbb{Z}}_{\ell}[G])$ is not, in general, a perfect complex. It belongs,
however, to a certain full subcategory of $D^b_c({\mathbb{Z}}_{\ell}[G])$ considered by Rickard \cite{Ri}.

Recall that a ${\mathbb{Z}}_{\ell}[G]$-module is called a \emph{permutation
module} if it is isomorphic to a module of the form ${\mathbb{Z}}_{\ell}[E]$ for a finite $G$-set $E$, in other words, if it admits
a finite basis over ${\mathbb{Z}}_{\ell}$ which is set-theoretically
stable under~$G$.
Denote by ${\mathbb{Z}}_{\ell}[G]_{\perm}$ the
full subcategory of the category of ${\mathbb{Z}}_{\ell}[G]$-modules consisting of
direct summands of permutation modules. This is an additive
subcategory. By a result of Rouquier \cite{Ro},
the natural functor $K^b({\mathbb{Z}}_{\ell}[G]_{\perm}) \rightarrow D^b_c({\mathbb{Z}}_{\ell}[G])$ induces a fully faithful functor
\beq
K^b({\mathbb{Z}}_{\ell}[G]_{\perm})/K^b_0({\mathbb{Z}}_{\ell}[G]_{\perm})
\rightarrow D^b_c({\mathbb{Z}}_{\ell}[G]), \label{(3.9.1)}
\eeq
where $K^b_0({\mathbb{Z}}_{\ell}[G]_{\perm})$ denotes the full subcategory
of $K^b({\mathbb{Z}}_{\ell}[G]_{\perm})$
consisting of acyclic complexes. In particular, the essential image
$D^b({\mathbb{Z}}_{\ell}[G])_{\perm}$ of \eqref{(3.9.1)} is a triangulated
subcategory of $D^b_c({\mathbb{Z}}_{\ell}[G])$.
\end{ssect}

\begin{Theorem}\label{3.10}
Under the assumptions of 3.1, denote by $D^b({\mathbb{Z}}_{\ell}[G])_{X\hperm}$ the smallest full triangulated subcategory of
$D^b({\mathbb{Z}}_{\ell}[G])_{\perm}$ containing the direct summands of
permutation modules
of the form ${\mathbb{Z}}_{\ell}[G/I]$, where $I$ runs through the inertia
groups of~$G$. Then $R\Gamma_c(X,{\mathbb{Z}}_{\ell})$ (resp. $R\Gamma(X,{\mathbb{Z}}_{\ell})$, if $X$ is separated,) belongs to $D^b({\mathbb{Z}}_{\ell}[G])_{X\hperm}$.
\end{Theorem}

\begin{Remark}\label{3.11}
(a) For the result on $R\Gamma(X,\Z_\ell)$, the assumption that $X$ is separated serves only to guarantee the existence of $X/G$. One may replace it by the weaker assumption that the inertia subgroup $I(G,X)$ of $G\times X$ is finite over $X$.

(b) When $G$ acts freely, $D^b({\mathbb{Z}}_{\ell}[G])_{X\hperm}$ is
the category of perfect complexes of
${\mathbb{Z}}_{\ell}[G]$-modules: one recovers \ref{3.1} (b).

(c) When $X$ is a separated scheme and $G$ acts admissibly, the
result of \ref{3.10} for $R\Gamma_c(X,{\mathbb{Z}}_{\ell})$ is a
weak form of Rickard's theorem \cite[3.2]{Ri}. Rickard actually
constructs a bounded complex of finite sums of direct summands of
permutation modules of the form ${\mathbb{Z}}_{\ell}[G/I]$, where
$I$ is an inertia group of $G$, well defined up to homotopy,
representing $R\Gamma_c(X,{\mathbb{Z}}_{\ell})$. Rickard does not
consider the case of $R\Gamma$.
\end{Remark}

\begin{ssect}\label{3.12}
Let us introduce some notations before proving \ref{3.10}. For a
homomorphism of finite groups $f\colon G\to H$, the contracted
product is defined to be $X\wedge^G H = (X\times H)/G$, where $G$
acts on $X\times H$ by $(x,h)g=(xg, f(g)^{-1}h)$. See \cite[III
1.3.1]{Giraud} for a definition in a topos. The right translation of
$H$ induces an action of $H$ on $X\wedge^G H$. We have $(X\wedge^G
H)\wedge^H \{1\}\simeq X\wedge^G \{1\}$, namely $(X\wedge^G
H)/H\simeq X/G$. If $f$ is a monomorphism, the projection $X\wedge^G
H\to X$ is a finite \'etale cover of fibers isomorphic to $H/G$.

As in the proof of \ref{3.3}, we denote by $X_H$ the open subset of
$X^H$ which is the complement of the union of the $X^{H'}$ for $H'$
strictly containing $H$. The stabilizer of any geometric point in
$X_H$ is $H$. Let ${\cal S}$ be the set of conjugacy classes of
subgroups of $G$. For $S \in {\cal S}$, denote by $X_S$ the
(disjoint) union of the $X_H$'s for $H \in S$, with its induced
action of $G$, and $Y_S = X_S/G$. For $H \in S$, the normalizer
$N_G(H)$ acts on $X_H$ through $N_G(H)/H$. The projection $X_H
\rightarrow X_H/N_G(H) = Y_S$ is an \'{e}tale Galois cover of group
$N_G(H)/H$, and $X_S = X_H \wedge^{N_G(H)} G$.  The $Y_S$'s, for $S
\in {\cal S}$, form a decomposition of $Y=X/G$ into disjoint locally
closed algebraic subspaces of $Y$.
\end{ssect}

\begin{proof}[Proof of \ref{3.10}]
Let us first treat the case of $R\Gamma_c$. We may assume $X$
separated. We proceed by Noetherian induction on $X$. There is a
$G$-stable dense open subset $V$ of $X$ which is a disjoint union of
irreducible schemes. Take one component $W$ of $V$, and let $H$ be
the inertia group at the generic point. Then $W^H=W$ as sets. Let
$U=W_S$, where $S$ is the conjugacy class of $H$. Then $U$ is a
nonempty $G$-stable open subset of $V$, disjoint union of the
$U_{H'}=U^{H'}$, for $H'$ running through the conjugacy class $S$,
such that $G$ acts transitively on the maximal points of $U$. Up to
shrinking~$U$, we may assume $U$ affine. We have a distinguished
triangle in $D^b_c({\mathbb{Z}}_{\ell}[G])$:
\[
R\Gamma_c(U,{\mathbb{Z}}_{\ell})  \rightarrow R\Gamma_c(X,{\mathbb{Z}}_{\ell})
\rightarrow R\Gamma_c(Y,{\mathbb{Z}}_{\ell}) \rightarrow,
\]
where $Y$ is the complement of $U$ in $X$.
By the induction hypothesis, we may therefore assume that $X =
U$. The group $N_G(H)$ acts on $X_H$. The natural
map (of $D^b({\mathbb{Z}}_{\ell}[G])$)
\[
{\mathbb{Z}}_{\ell}[G] \otimes^L_{{\mathbb{Z}}_{\ell}[N_G(H)]} R\Gamma_c(X_H,{\mathbb{Z}}_{\ell}) \rightarrow R\Gamma_c(X,{\mathbb{Z}}_{\ell})
\]
is an isomorphism, as $X = G \wedge^{N_G(H)} X_H$. As $N_G(H)$ acts
on $X_H$ through $N_G(H)/H$ and the action of $N_G(H)/H$ is free,
$R\Gamma_c(X_H,{\mathbb{Z}}_{\ell})$ is perfect over
${\mathbb{Z}}_{\ell}[N_G(H)/H]$ (3.1 (b)). As ${\mathbb{Z}}_{\ell}[G]
\otimes^L_{{\mathbb{Z}}_{\ell}[N_G(H)]} {\mathbb{Z}}_{\ell}[N_G(H)/H]
= {\mathbb{Z}}_{\ell}[G/H]$, it follows that
$R\Gamma_c(X,{\mathbb{Z}}_{\ell})$ is represented by a bounded
complex of finite sums of direct summands of
${\mathbb{Z}}_{\ell}[G/H]$, which finishes the proof in this case.

The proof for the case of $R\Gamma$ is similar.
Consider the open immersion $u \colon U \rightarrow X$ as above, and its
complement $i \colon Y \rightarrow X$. The exact sequence $0 \rightarrow
u_!{\mathbb{Z}}_{\ell} \rightarrow {\mathbb{Z}}_{\ell} \rightarrow i_*{\mathbb{Z}}_{\ell} \rightarrow 0$ gives an exact triangle in $D^b_c({\mathbb{Z}}_{\ell}[G])$:
\[
R\Gamma(X,u_!{\mathbb{Z}}_{\ell})  \rightarrow R\Gamma(X,{\mathbb{Z}}_{\ell})
\rightarrow R\Gamma(Y,{\mathbb{Z}}_{\ell}) \rightarrow .
\]
By the induction hypothesis it suffices to show that $R\Gamma(X,u_!{\mathbb{Z}}_{\ell})$ belongs to $D^b({\mathbb{Z}}_{\ell}[G])_{X\hperm}$. Consider
the commutative diagram
\beq
\xymatrix{U \ar[r]^u \ar[d]_{f_U} & X \ar[d]^f\\
U/G \ar[r]^v & X/G }, \label{(3.12.1)}
\eeq
where $f$ is the canonical projection. We have, in $D^b_c({\mathbb{Z}}_{\ell}[G])$ (by \ref{d.fin}),
\beq
R\Gamma(X,u_!{\mathbb{Z}}_{\ell}) = R\Gamma(X/G, f_* u_! \Z_\ell) = R\Gamma(X/G,v_!(f_U)_*{\mathbb{Z}}_{\ell}). \label{(3.12.2)}
\eeq
As above, let $H$ be a member of $S$, and let $f_H \colon U_H \rightarrow
U/G$ be the restriction of $f_U$ to $U_H$. We have $(f_U)_*{\mathbb{Z}}_{\ell} = {\mathbb{Z}}_{\ell}[G] \otimes_{{\mathbb{Z}}_{\ell}[N_G(H)]}
(f_H)_*{\mathbb{Z}}_{\ell}$, hence $v_!(f_U)_*{\mathbb{Z}}_{\ell} = {\mathbb{Z}}_{\ell}[G] \otimes_{{\mathbb{Z}}_{\ell}[N_G(H)]}
v_!(f_H)_*{\mathbb{Z}}_{\ell}$, so that by the projection formula, we get
\beq
R\Gamma(X/G,v_!(f_U)_*{\mathbb{Z}}_{\ell}) = {\mathbb{Z}}_{\ell}[G]
\otimes^L_{{\mathbb{Z}}_{\ell}[N_G(H)]} R\Gamma(X/G,v_!(f_H)_*{\mathbb{Z}}_{\ell}). \label{(3.12.3)}
\eeq
As $N_G(H)$ acts on
$U_H$ through $N_G(H)/H$ and the action of $N_G(H)/H$ is free,
$(f_H)_*{\mathbb{Z}}_{\ell}$ is locally isomorphic to ${\mathbb{Z}}_{\ell}[N_G(H)/H]$. Therefore $v_!(f_H)_*{\mathbb{Z}}_{\ell}$ is of
finite tor-dimension over $N_G(H)/H$, and consequently,
$R\Gamma(X/G,v_!(f_H)_*{\mathbb{Z}}_{\ell})$ is perfect over $N_G(H)/H$.
Then the conclusion follows from \eqref{(3.12.2)} and \eqref{(3.12.3)} by the same
argument as for
the case of $R\Gamma_c$.
\end{proof}

\section{Tameness at infinity}
In this section we fix an algebraically closed field $k$ of
characteristic exponent $p$ and a prime number $\ell \ne p$.

\begin{ssect}\label{4.1} Let $Y$ be a normal, connected scheme, separated and of finite
type over $k$, $G$ a finite group, and $f \colon X \rightarrow Y$ an
\'{e}tale Galois cover of group $G$. We have $Y=X/G$. Let
$\overline{Y}$ be a normal compactification of $Y$ over~$k$ (\ie a
proper, normal, connected scheme over $k$ containing $Y$ as a dense
open subset). We say that $X$ is \emph{tamely ramified along
$\overline{Y} - Y$} if $G$ acts \emph{tamely} on the normalization
$\overline{X}$ of $\overline{Y}$ in $X$, \ie the inertia groups
$G_{\overline{x}} \subset G$, stabilizers of geometric points
$\overline{x}$ of $\overline{X}$ in $G$, are of order prime to $p$
(cf. \cite[2.6]{I1}). We say that $X$ is \emph{tamely ramified at
infinity over $Y/k$} if there exists a normal compactification
$\overline{Y}$ of $Y$ such that $X$ is tamely ramified along
$\overline{Y} - Y$. By \cite[2.8]{I1}, \eqref{(3.2.1)} still holds
in this case, namely, we have an equality in the Grothendieck group
$P_{\ell}(G) = K^{\bullet}({\mathbb{Z}}_{\ell}[G])$ of finitely
generated projective ${\mathbb{Z}}_{\ell}[G]$-modules ($=
K^{\bullet}({\mathbb{F}}_{\ell}[G])$): \beq
\chi_c(X,G,{\mathbb{Z}}_{\ell}) =
\chi(X/G)\Reg_{{\mathbb{Z}}_{\ell}}(G), \label{(4.1.1)} \eeq or,
equivalently, \beq \chi_c(X,G,{\mathbb{F}}_{\ell}) =
\chi(X/G)\Reg_{{\mathbb{F}}_{\ell}}(G), \label{(4.1.2)} \eeq where
$\chi_c(X,G,{\mathbb{Z}}_{\ell}) = \chi_c(X,G,{\mathbb{F}}_{\ell})$
is the class of $R\Gamma_c(X,G,{\mathbb{Z}}_{\ell})$ (or
$R\Gamma_c(X,G,{\mathbb{F}}_{\ell})$) in $P_{\ell}(G)$, and $\Reg$
denotes a regular representation. Note that, as the natural
homomorphism $P_{\ell}(G) \rightarrow R_{{\mathbb{Q}}_{\ell}}(G)$ is
injective, we have \beq \chi_c(X,G,{\mathbb{Z}}_{\ell}) =
\chi(X,G,{\mathbb{Z}}_{\ell}) = \chi(X,G,{\mathbb{F}}_{\ell}) =
\chi_c(X,G,{\mathbb{F}}_{\ell}), \label{(4.1.3)} \eeq where
$\chi(X,G,{\mathbb{Z}}_{\ell})$ (resp.
$\chi(X,G,{\mathbb{F}}_{\ell})$) is the class of
$R\Gamma(X,G,{\mathbb{Z}}_{\ell})$ (resp.
$R\Gamma(X,G,{{\mathbb{F}}_{\ell}})$).
\end{ssect}

\begin{ssect}\label{4.2} We will reformulate the notion of tameness at infinity in terms
of \emph{Vidal's group} $K(Y,{\mathbb{F}}_{\ell})^0_t$ \cite[2.3.1]{Vi1} (where the group is denoted by $K_c(Y,{\mathbb{F}}_{\ell})^0_t$).

Let us briefly recall its definition. Let $Z$ be a scheme separated
and of finite type over~$k$. We denote by $K(Z,{\mathbb{F}}_{\ell})$ the Grothendieck group of constructible ${\mathbb{F}}_{\ell}$-modules on~$Z$, and by $K_{\lisse}(Z,{\mathbb{F}}_{\ell})$ the
subgroup generated by the classes of lisse ${\mathbb{F}}_{\ell}$-sheaves,
and by $[\mathcal{F}]$ the class in $K(Z,{\mathbb{F}}_{\ell})$ of a constructible
sheaf $\mathcal{F}$.
When $Z$ is normal, connected, with geometric generic point
$\zeta$, $K_{\lisse}(Z,{\mathbb{F}}_{\ell})$ is the Grothendieck group of
finite, continuous ${\mathbb{F}}_{\ell}[\pi_1(Z,\overline{\zeta})]$-modules. In this case,
for $a \in K_{\lisse}(Z,{\mathbb{F}}_{\ell})$, we
denote by
\[
\trBr(-,a) \colon \pi_1(Z,\zeta) \rightarrow {\mathbb{Z}}_{\ell},
\ g \mapsto \trBr(g,a_{\zeta})
\]
its Brauer trace (on $\ell$-regular elements, extended by zero on
$\ell$-singular ones). Vidal defines a closed subset
\[
E_{Z/k} \subset \pi_1(Z,\zeta),
\]
stable under conjugation, called the \emph{wild part} of
$\pi_1(Z,\zeta)$: this is the intersection, over all normal
compactifications $\overline{Z}$ of $Z$, of the closures
$E_{Z/k,\overline{Z}}$ of the unions of the conjugates of the images
of the pro-$p$-Sylow subgroups of the local fundamental groups
$\pi_1(\overline{Z}_{(y)}\times_{\overline{Z}} Z,z)$, where $y$ runs
through geometric points of $\overline{Z}-Z$ and $z$ is a geometric
point of $\overline{Z}_{(y)}\times_{\overline{Z}} Z$ above $\zeta$
(see \cite[2.1]{Vi1}, \cite[1.1]{Vi2} for more details). Now, for
$Z$ only assumed to be separated and of finite type over $k$, Vidal
defines
\[
K(Z,{\mathbb{F}}_{\ell})^0_t \subset K(Z,{\mathbb{F}}_{\ell})
\]
as the subgroup generated by the classes $i_!a$, where $i \colon Y \rightarrow
Z$ is separated and quasi-finite, with $Y$ normal connected, and $a \in K_{\lisse}(Y,\F_\ell)$
has the property that, for all $g \in E_{Y/k}$, $\trBr(g,a) = 0$.

We extend this definition to algebraic spaces. More precisely, for an algebraic space $Z$ of finite type over $k$, we denote by $K(Z,\F_\ell)$ the Grothendieck group of constructible $\F_\ell$-modules on $Z$, and we define
\[
K(Z,{\mathbb{F}}_{\ell})^0_t \subset K(Z,{\mathbb{F}}_{\ell})
\]
as the subgroup generated by the classes $i_!a$, where $i \colon Y \rightarrow
Z$ is quasi-finite, with $Y$ a separated normal connected scheme, and $a \in K_{\lisse}(Y,{\mathbb{F}}_{\ell})$
has the property that, for all $g \in E_{Y/k}$, $\trBr(g,a) = 0$. This definition does not depend on the choice of geometric points.

Recall that, when $Z$ is a separated normal connected scheme, it follows from a
valuative criterion
of Gabber that, for $a \in K_{\lisse}(Z,{\mathbb{F}}_{\ell})$, $a$ belongs to $K(Z,\F_\ell)^0_t$ if and only if there exists a normal
compactification $\overline{Z}$ of $Z$ such that, for all $g \in
E_{Z/k,\overline{Z}}$, $\trBr(g,a) =0$
\cite[6.2 (ii)]{Vi2}.
\end{ssect}

The following is a variant of \cite[2.3.3]{Vi1} and \cite[0.1]{Vi2}:

\begin{Proposition}\label{4.2.1}
Let $f\colon Z \to W$ be a morphism of algebraic spaces of finite type over~$k$. Then
\begin{enumerate}
\item The map $f^*\colon K(W, \F_\ell)\to K(Z, \F_\ell)$ sends $K(W, \F_\ell)^0_t$ into $K(Z, \F_\ell)^0_t$.

\item The map $f_!\colon K(Z, \F_\ell)\to K(W, \F_\ell)$ sends $K(Z, \F_\ell)^0_t$ into $K(W, \F_\ell)^0_t$.

\item If $Z=\coprod_I Z_i$ is a partition of $Z$ into locally closed algebraic subspaces, then $\phi\colon a \mapsto (a | Z_i)_I$ defines an isomorphism $K(Z, \F_\ell)^0_t \simeq \oplus_I K(Z_i, \F_\ell)^0_t$.

\item $K(Z,\F_\ell)^0_t$ is an ideal of the ring $K(Z,\F_\ell)$.
\end{enumerate}
\end{Proposition}

\begin{proof}
(a) Let $X$ be a separated connected normal scheme of finite type over $k$, $i\colon X\to W$ be a quasi-finite morphism, $a\in K_{\lisse}(X,\F_\ell)$ satisfying $\trBr(g,a) = 0$ for all  $g \in E_{X/k}$. Then, by base change,
\beq\label{e.Vbc}
f^*i_! a = \sum_{j\in J} i_{Y_j!} f_j^* a,
\eeq
where $X\times_W Z = \coprod_{j\in J} Y_j$ is a partition into locally closed separated normal connected subschemes, $i_{Y_j}\colon Y_j \to Z$ and $f_j\colon Y_j\to X$ are the projections. Note that $f_j^* a\in K_\lisse(Y_j,\F_\ell)$. If we still denote by $f_j$ the map $E_{Y_j,k}\to E_{X,k}$ induced by $f_j$ \cite[2.1.1]{Vi1}, then
\[\trBr(g,f_j^* a)=\trBr(f_j(g),a)=0.\]
Thus $f^*i_! a$ belongs to $K(W, \F_\ell)^0_t$ by \eqref{e.Vbc}.

(c) By (a), the homomorphism $\phi$ in (c) is well defined. We define a homomorphism $\psi\colon \oplus_I K(Z_i, \F_\ell)^0_t \to K(Z, \F_\ell)^0_t$ by $(a_{Z_i})_I\mapsto \sum_I i_{Z_i !} a_{Z_i}$, where $i_{Z_i}$ is the immersion $Z_i \to Z$. The two homomorphisms are clearly inverse to each other.

(b) If $f$ is quasi-finite, (b) follows from the definition. For the general case, applying the quasi-finite case and (c), we may reduce to the case where $f$ is morphism of separated schemes. In this case (b) is \cite[0.1]{Vi2}.

(d) Let $i_X\colon X\to Z$ be a quasi-finite morphism where $X$ is a
separated normal connected scheme, let $a\in K_\lisse(X,\F_\ell)$
satisfying $\trBr(g,a)=0$ for all $g\in E_{X/k}$, and let $b\in
K(Z,\F_\ell)$. By projection formula, \beq\label{e.Vpf}
(i_!a)b=i_!(a(i^* b))= \sum_{j\in J} i_{X_j!}((a|X_j)(b|X_j)), \eeq
where $X=\coprod_{j\in J} X_j$ is a partition into locally closed
normal connected subschemes such that $b|X_j\in
K_\lisse(X_j,\F_\ell)$, $i_{X_j}\colon X_j\to Z$ is the composition
with $i$. As in (a), it follows from the functoriality of $E$ that
$\trBr(g,a|X_j)=0$ for all $g\in E_{X_j/k}$. Thus
$\trBr(g,(a|X_j)(b|X_j))=0$ for all $g\in E_{X_j/k}$ by the
multiplicativity of the Brauer trace \cite[18.1 iv)]{Se1}. Therefore,
$(i_!a)b$ belongs to $K(Z,\F_\ell)^0_t$ by \eqref{e.Vpf}.
\end{proof}


The rank function on constructible ${\mathbb{F}}_{\ell}$-sheaves defines a
ring homomorphism
\[
\rk \colon  K(Z,{\mathbb{F}}_{\ell}) \rightarrow C(Z,{\mathbb{Z}}),
\]
where $C(Z,{\mathbb{Z}})$ is the ring of constructible functions
on $Z$ with values in ${\mathbb{Z}}$. This homomorphism has a
natural section which is a ring homomorphism, associating with a function $c \in C(Z,{\mathbb{Z}})$
the class $\langle c\rangle$ of the constructible sheaf $\oplus j_i{}_!
\F_{\ell}^{c | Z_i}$, where $Z$ is a disjoint union of locally
closed subspaces $j_i \colon Z_i \rightarrow Z$ over which $c$ is constant
($\langle c\rangle$ is independent of the choice of the stratification). Note that $K(Z,\F_\ell)^0_t$ is contained in $\Ker(\rk)$.

\begin{Definition}\label{4.3}
For an algebraic space $Z$ separated and of finite type over $k$ and $a \in K(Z,{\mathbb{F}}_{\ell})$
we will say that $a$ is \emph{virtually tame} if $a - \langle\rk(a)\rangle$
belongs to $K(Z,{\mathbb{F}}_{\ell})^0_t$.
\end{Definition}

We denote by $K(Z,\F_\ell)_t$ the subgroup of $K(Z,\F_\ell)$ consisting of virtually tame elements. As a subgroup, it is generated by $K(Z,\F_\ell)^0_t$ and the image of $\langle - \rangle$. It follows that $K(Z,\F_\ell)_t$ is a subring of $K(Z,\F_\ell)$. The rank function induces a ring isomorphism
\[K(Z,\F_\ell)_t/K(Z,\F_\ell)^0_t\longsimto C(Z,\F_\ell).\]

\begin{Remark}\label{4.3.1} When $Z$ is a normal connected scheme, it follows from the
consequence of the valuative criterion
of Gabber mentioned at the end of \ref{4.2} that, for $a \in K_{\lisse}(Z,{\mathbb{F}}_{\ell})$, $a$ is
virtually tame if and only if there exists a normal
compactification $\overline{Z}$ of $Z$ such that, for all $g \in
E_{Z/k,\overline{Z}}$, $\trBr(g,a) = \rk(a)$.
\end{Remark}

The following is an immediate consequence of \ref{4.2.1}.

\begin{Proposition}\label{p.t}
Let $f\colon Z\to W$ be a morphism of algebraic spaces of finite type over~$k$. Then
\begin{enumerate}
\item The map $f^*\colon K(W,\F_\ell)\to K(Z,\F_\ell)$ sends $K(W,\F_\ell)_t$ to $Z(W,\F_\ell)_t$.

\item If $Z=\coprod_I Z_i$ is a partition of $Z$ into locally closed algebraic subspaces, then $\phi\colon a \mapsto (a | Z_i)_I$ defines an isomorphism $K(Z, \F_\ell)_t \simeq \oplus_I K(Z_i, \F_\ell)_t$.

\item If $a,b\in K(Z,\F_\ell)$ such that $a$ and $ab$ are virtually tame, and $\rk(a)$ is invertible, then $b$ is virtually tame.
\end{enumerate}
\end{Proposition}

\begin{Proposition}\label{4.4}
Let $Y$ be a normal, connected scheme, separated and of finite type
over $k$, with a generic geometric point $y$, and let
$F$ be a lisse ${\mathbb{F}}_{\ell}$-sheaf on $Y$. Denote by $\rho \colon
\pi_1(Y,y) \rightarrow \Aut(\cF_{y})$
the representation defined by $\cF$. The following
conditions are equivalent:
\begin{enumerate}
\item $[\cF]$ is virtually tame (\ref{4.3}).

\item There exists a normal compactification $\overline{Y}$ of $Y$ over
$k$ such that, for all $g \in E_{Y/k,\overline{Y}}$, $\rho(g) = 1$.
\end{enumerate}
\end{Proposition}

\begin{proof}
The implication (b) $\Rightarrow$ (a) is trivial. Conversely, by \ref{4.3.1}, (a) implies the existence of a normal
compactification $\overline{Y}$ of $Y$ over $k$ such that for all $g\in E_{Z/k, \overline{Z}}$, $\trBr(g,a)=\rk(a)$.
The representation $\rho$ factors through $\rho'\colon G\to \Aut(\cF_y)$ for some finite quotient $G$ of $\pi_1(Y,y)$. As in \cite[2.1.1]{Vi1}, let $E_{Y/k,\overline{Y}}(G)$
be the image of $E_{Y/k,\overline{Y}}$ in $G$. By \cite[18.2, Cor. 1]{Se1}, the restriction of $\rho'$ to any subgroup of $G$ contained in $E_{Y/k,\overline{Y}}(G)$ is
the trivial representation. Thus (b) follows from the fact that $E_{Y/k,\overline{Y}}(G)$ is a union of $p$-subgroups of $G$.
\end{proof}

In particular:
\begin{Corollary}\label{4.5}
Let $f \colon X \rightarrow Y$ be as in \ref{4.1}. The following conditions are
equivalent:
\begin{enumerate}
\item $[f_*{\mathbb{F}}_{\ell}]$ is virtually tame (\ref{4.3}).

\item $X$ is tamely ramified at infinity over $Y/k$ (\ref{4.1}).
\end{enumerate}
\end{Corollary}

\begin{ssect}\label{4.6}
Let $X$ be a an algebraic $k$-space of finite type, endowed
with an action of a finite group~$G$. Assume that the inertia $I(G,X)$ is finite over $X$ and let $f \colon X \rightarrow
Y = X/G$ be the projection. Here $Y$ is an algebraic space of finite type over $k$. We say that $G$ acts \emph{virtually
tamely} on $X$ if $[f_*{\mathbb{F}}_{\ell}]$ is virtually tame (\ref{4.3}).

To give a more concrete characterization of this property, we assume
for simplicity that $X$ is separated. We adopt the notations of
\ref{3.12}. For each $S \in {\cal S}$, write $Y_S$ as a finite
disjoint union of normal, connected, locally closed subschemes
$(Y_S)_i$, $i \in J_S$. Let $(f_S)_i \colon (X_S)_i \rightarrow
(Y_S)_i$ be the base change of $f \colon X \rightarrow Y$ to
$(Y_S)_i$. This is a disjoint sum of \'{e}tale Galois covers
$(X_{H})_i$ ($H \in S$) of $(Y_S)_i$ of group $N_G(H)/H$,
transitively permuted by $G$.
\end{ssect}

\begin{Proposition}\label{p.vt}
Using the above notations, $G$ acts virtually tamely on $X$ if and only
if, for all $S \in {\cal S}$, $H\in S$ and $i \in J_S$,
$(X_{H})_i$ is tamely ramified at infinity over $(Y_S)_i$.
\end{Proposition}

In particular, the condition that $G$ acts virtually tamely does not depend on
$\ell$.

\begin{proof}
As $f_*{\mathbb{F}}_{\ell} | (Y_S)_i =
(f_S)_{i*}{\mathbb{F}}_{\ell}$, $[f_*{\mathbb{F}}_{\ell}]$ is virtually tame if
and only if $[(f_S)_{i*}{\mathbb{F}}_{\ell}]$ is
virtually tame for all $S \in {\cal S}$ and $i \in J_S$ by \ref{p.t} (b). For all $H\in S$,
\[[(f_S)_{i*}\F_\ell]=( G:N_G(H) ) [(f_H)_{i*}\F_\ell],\]
where $(f_H)_i\colon (X_H)_i\to (Y_S)_i$ is the restriction of $(f_S)_i$. By \ref{p.t} (c), it follows  that $[(f_S)_{i*}\F_\ell]$ is virtually tame if and only if $[(f_H)_{i*}\F_\ell]$ is virtually tame.
We then apply \ref{4.5} to $(f_H)_i$.
\end{proof}

The following is a generalization of \ref{3.2}, which is an analogue of
Verdier's formula \cite[Th., p.~443]{V}.

\begin{Theorem}\label{4.7}
Let $X$ be an algebraic $k$-space, separated and of finite type, endowed with a virtually tame action of a finite group $G$. Then, with the
notations of \ref{4.6}, we have, in $R_{{\mathbb{Q}}_{\ell}}(G)$,
\beq
\chi(X,G,{\mathbb{Q}}_{\ell}) = \sum_{S \in {\cal S}} \chi(X_S/G)I_S,  \label{(4.7.1)}
\eeq
where $I_S = [{\mathbb{Q}}_{\ell}[G/H]] \in R_{{\mathbb{Q}}_{\ell}}(G)$ for
$H \in S$.
\end{Theorem}

By the comparison between \'{e}tale and Betti cohomologies (see the remark following \ref{3.4}), we recover Verdier's formula \cite[Th., p.~443]{V} (a
generalization of \ref{3.4}):

\begin{Corollary}\label{4.8}
Let $X$ be an algebraic space separated and of finite type over ${\mathbb{C}}$
endowed with an action of a finite group $G$. Then, with the above
notations, we have, in $R_{{\mathbb{Q}}}(G)$,
\beq
\chi(X,G,{\mathbb{Q}}) = \sum_{S \in {\cal S}} \chi(X_S/G,{\mathbb{Q}})I_S,  \label{(4.8.1)}
\eeq
where $I_S = [{\mathbb{Q}}[G/H]] \in R_{{\mathbb{Q}}}(G)$
for $H \in S$.
\end{Corollary}

\begin{proof}[Proof of \ref{4.7}]
By the equality between $\chi$ and $\chi_c$ and
the additivity of $\chi_c$,
\[\chi(X,G,\Q_\ell)=\sum_{S\in \cS} \chi(X_S,G,\Q_\ell).\]
Thus we may assume $X = X_S$ for some $S\in \cS$. By the additivity
of $\chi$ again, we may assume $X/G$ is a normal connected scheme.
For $H\in S$, by \ref{p.vt}, $X_H$ is tamely ramified at infinity
over $X/G$. Then \eqref{(4.1.1)} gives
\[\chi(X,N_G(H),\Q_\ell)=\chi(X/G)[\Q_\ell[N_G(H)/H]].\]
As $X = X_H\wedge^{N_G(H)} G$, we have
\[\chi(X,G,\Q_\ell)=\chi(X/G)[\Q_\ell[G/H]].\]
\end{proof}

The following application of \ref{4.7} is a generalization of a result of
Petrie-Randall \cite[3.1]{PR} and of \ref{3.3}:

\begin{Corollary}\label{4.9}
Under the assumptions of \ref{4.7}, suppose that $G$ is cyclic, generated by~$g$. Then, with the notations of \ref{s.space},
\[
t(g) = \chi(X^G).
\]
\end{Corollary}

Indeed, by \eqref{(4.7.1)} we have
\[
t(g) = \sum_{S \in {\cal S}}\chi(X_S/G)\Tr(g,I_S).
\]
Now, $\Tr(g,I_S) = 0$ unless $S = \{G\}$, in which case $X_S =
X^G$ and $\Tr(g,I_S) = 1$.

\section{The case of rigid cohomology}
The results of this section will not be used in the sequel.

\begin{ssect}\label{5.1} We use the notations of \ref{s.Laumon}.  We assume $p > 1$ and $k$
algebraically closed.  Following \cite[8.2.4]{Ls}, we define a
\emph{realization} of $X$ to be a sequence of immersions
$X\stackrel{j}{\hookrightarrow} \overline{X}\hookrightarrow P$, where $j$ is an open immersion, $\overline{X}$ is a proper $k$-scheme and $P$ is
a formal scheme over $W=W(k)$, smooth in a neighborhood of $X$. We
say $X$ is \emph{realizable} if such a realization exists. This is
the case if $X$ is quasi-projective. Indeed, if $X\subset
\mathbb{P}^n_k$, we can take $\overline{X}$ to be the closure of $X$ in $\mathbb{P}^n_k$ and
$P=\widehat{\mathbb{P}^n_{W}}$. Given a realization
$X\hookrightarrow \overline{X}\hookrightarrow P$, we can construct
$G$-equivariant immersions
\[X\hookrightarrow X'=\prod_{g\in
G}\overline{X} \hookrightarrow \prod_{g\in
G}P\]
as in \cite[3.6]{Z}. Here $G$ acts on the products by permutation of the factors. We obtain a $G$-equivariant realization by taking the closure of $X$ in $X'$.

Let $K$ be the fraction field of $W$. For $X$ realizable, we denote
by $H^*_{c,\rig}$ the rigid cohomology with compact support of $X/K$
in the sense of Berthelot (\cite[3.1]{B1}, \cite[8.2.5]{Ls}). The
action of $G$ on $X$ defines, by functoriality, an action of $G$ on
$H^*_{c,\rig}(X/K)$. As the category of $K[G]$-modules is
semisimple, this action turns the complex $R\Gamma_{c,\rig}(X/K)$
defining $H^*_{c,\rig}(X/K)$ into an object of $D^b(K[G])$. Such a
complex can also be defined directly as follows. Choose a
$G$-equivariant realization $X\hookrightarrow
\overline{X}\hookrightarrow P$. Then we have
\[R\Gamma_{c,\rig}(X/K)= R\Gamma(]\overline{X}[_P,R\underline{\Gamma}_{]X[_P}(\Omega^\bullet_{]\overline{X}[_P})).\]

By Berthelot's finiteness theorem \cite[3.9 (i)]{B2}, this complex
has finite-dimensional cohomology groups, and, when $X/k$ is
\emph{proper and smooth}, is isomorphic to $R\Gamma(X/W) \otimes_W^L
K$, where $R\Gamma(X/W) \in D^b(W[G])$ is the complex calculating the
crystalline cohomology of $X/W$ \cite[1.9]{B2}. Recall that, for any
proper and smooth $k$-scheme $X$, $R\Gamma(X/W)$ can be computed by
$R\Gamma(X,W\Omega^\bullet_X)$, where $W\Omega^\bullet_X$ is the de
Rham-Witt complex of $X/k$ \cite[I 1.15, II (2.8.2)]{IllRW}, which is
a complex of $G$-$W$-modules on (the Zariski site of) $X$. Note that
here the category of $W[G]$-modules is no longer semisimple, and this
complex can't be recovered from the mere datum of its cohomology
groups, the $W[G]$-modules $H^i(X/W)$.

One can't expect
in general that, if $G$ acts freely, $R\Gamma_{c,\rig}(X/K)$ comes
by extension of scalars from a perfect complex of $W[G]$-modules.
Indeed, if it were the case, the traces of $p$-singular elements
would be zero \cite[16.2, Th.~36]{Se1}, and in the example given
after \eqref{(1.8.1)}, the trace of $s$ on $H^*_{c,\rig}(X/K)$ can
be shown to be equal to 1. We have, however, the following results,
which complement \ref{3.2}:
\end{ssect}

\begin{Theorem}\label{5.2}
Let $X/k$ with the action of $G$ be as in \ref{1.1}, with $p >1$. With the
notations of \ref{5.1}:
\begin{enumerate}
\item If $X$ is realizable, $\Tr(s, R\Gamma_{c,\rig}(X/K)) := \sum
(-1)^i\Tr(s,H^i_{c,\rig}(X/K))$ is equal to the integer
$t(s)$ of \ref{3.1} (i), for all $s \in G$.

\item If $X/k$ is proper and smooth, and $G$ acts freely on $X$, then
$R\Gamma(X/W)$ is a perfect complex of $W[G]$-modules.
\end{enumerate}
\end{Theorem}

For the proof we need the following well known lemmas:

\begin{Lemma}\label{5.3}
Let $X/k$ be a projective and smooth scheme, $s$ a $k$-endomorphism
of~$X$, $\ell$ a prime $\ne p$. Then we have an equality
of rational integers:
\beq\label{e.rigid}
\Tr(s,H^*(X/W) \otimes K) = \Tr(s,H^*(X,{\mathbb{Q}}_{\ell})).
\eeq
\end{Lemma}

\begin{proof}
Indeed, if $\CH^*(X)$ is the Chow ring of $X$, the theory of cycle
classes in $\ell$-adic (resp. crystalline) cohomology (\cite[Cycle]{SGA4d}, \cite{L1}) (resp. \cite{Gros}) gives a homomorphism $\CH^*(X)
\rightarrow H^*(X,{\mathbb{Q}}_{\ell})$ (resp. $\CH^*(X) \rightarrow
H^*(X/W) \otimes K)$), which is multiplicative and compatible with
Gysin maps. This implies that both sides of \eqref{e.rigid} are equal to the
intersection number of the graph of $s$ and the diagonal in $X \times
X$.
\end{proof}

\begin{Lemma}\label{l.uh}
  Let $f\colon X\to Y$ be a finite universal homeomorphism between
  $k$-schemes (resp.\ realizable $k$-schemes) separated of finite type. Then the canonical homomorphism
  \[R\Gamma_{\rig}(Y/K)\to R\Gamma_{\rig}(X/K),\quad\text{(resp.\ $R\Gamma_{c,\rig}(Y/K)\to R\Gamma_{c,\rig}(X/K)$)}\]
  is an isomorphism.
\end{Lemma}

\begin{proof}
  If $f$ is a bijective immersion, the assertions follow from the
  definitions. In the general case, the diagonal morphism $X\to \cosk_0
  (X/Y)_n$ is a bijective immersion. Thus the assertion for
  $R\Gamma_{\rig}$ follows from cohomological descent for finite
  morphisms \cite[4.5.1]{Tsuzuki}. For $R\Gamma_{\rig,c}$, we
  proceed by induction on $\dim Y$. For any closed subscheme $V$ of $Y$, we have a distinguished
  triangle \cite[3.1]{B1}
  \[R\Gamma_{c,\rig}((Y-V)/K)\to R\Gamma_{c,\rig}(Y/K)\to R\Gamma_{c,\rig}(X/K)\to.\]
  Replacing $Y$ by
  $Y_{\red}$ and shrinking $Y$, we may thus assume $Y$ quasi-projective, normal and integral.
  Let $\overline{Y}$ be
  a normal projective compactification of $Y$. Factorize $X\to \overline{Y}$ into a
  dense open immersion and a finite morphism:
  \[X\hookrightarrow \overline{X}\xrightarrow{\overline{f}} \overline{Y}.\]
  Then $\overline{f}$ is a universal homeomorphism. We are thus reduced to the case
  $Y$ projective, where the assertions for $R\Gamma_{c,\rig}$ and for $R\Gamma_{\rig}$ coincide.
\end{proof}

\begin{Lemma}\label{l.tor}
  Let $a<b$ be integers, $C\in D^{[a,b]}(W[G])$. Assume that the tor-amplitude
  of $C\otimes^L_W k\in D(k[G])$ is
  contained in $[a+1,b]$.
  Then the tor-amplitude of $C$ is contained in $[a,b]$.
\end{Lemma}

\begin{proof}
  The short exact sequence
  \begin{equation}\label{e.W}
  0\to W\xrightarrow{\times p} W \to k \to 0
  \end{equation}
  is resolution of $k$ by free $W$-modules.
  Let $M$ be a
  right $W[G]$-module. Then $M\otimes^L_W k\in D^{[-1,0]}$. Moreover $-\otimes_{K[G]}-$ is an exact bifunctor. Thus
  \begin{gather*}
  (M\otimes^L_{W[G]} C)\otimes^L_W k\simeq (M\otimes^L_W k)\otimes^L_{k[G]} (C\otimes^L_W k)\in
  D^{[a,b]}(k[G]),\\
  (M\otimes^L_{W[G]} C)\otimes_W K\simeq (M\otimes_W K)\otimes_{K[G]} (C\otimes_W K)\in
  D^{[a,b]}(K[G]).
  \end{gather*}
  Putting $E=M\otimes^L_{W[G]} C$ for brevity, \eqref{e.W} induces the long exact sequence
  \[H^{q-1}(E\otimes_W^L k)\to H^q(E)\xrightarrow{\times p} H^q(E)\to H^q(E\otimes_W^L k).\]
  For $q<a$, $H^q(E)\xrightarrow{\times p} H^q(E)$ is thus an isomorphism. Moreover, $H^q(E)\otimes_W K=0$. Thus $H^q(E)=0$.
\end{proof}

We prove \ref{5.2} (a) by induction on the dimension $d$ of~$X$,
using de Jong's Galois alterations, as in \cite[4.4]{Vi1} and
\cite[\S~3]{Z}. The assertion is trivial for $d = 0$. Assume $d \ge
1$. By \ref{l.uh}, we may assume $X$ reduced. Using the inductive
hypothesis and the additivity of traces on $H^*_{c,\rig}$, we may
replace $X$ by a dense open $G$-invariant subscheme. Therefore we
may assume $X$ smooth, affine. The connected components of~$X$ are
permuted by $s$, and the trace of $s$ is the sum of the traces of
$s$ on the cohomology of those components which are stabilized
by~$s$. So we may assume furthermore $X$ integral. Choose a
$G$-equivariant dense open embedding $j \colon X \rightarrow Z$,
with $Z/k$ a projective, integral $G$-scheme. By Gabber's refinement
of de Jong's results on equivariant alterations \cite[3.8]{Z} there
exist the following data:
\begin{itemize}
\item a surjective homomorphism $u \colon G' \rightarrow G$ of finite groups,

\item a projective smooth, integral $k$-scheme $Z'$ endowed with an
action of $G'$, and a surjective
$u$-equivariant $k$-morphism $a \colon Z' \rightarrow Z$,

\item a $G$-stable dense open affine subscheme $V$ of $X$,
\end{itemize}
satisfying the following property:
\begin{itemize}
\item if $H$ is the kernel of $u \colon G'
\rightarrow G$, and $V' = a^{-1}(V)$, $a|V  \colon V' \rightarrow V$ factors into
\[
V' \rightarrow V'/H \rightarrow V,
\]
where $V' \rightarrow V'/H$ is an \'{e}tale Galois cover of
group $H/H_0$, with $H_0 = H \cap \Ker(G' \rightarrow
\Aut(k(\eta'))$ ($\eta'$ the generic point of $X'$), and $V'/H
\rightarrow V$ is a finite and flat universal homeomorphism.
\end{itemize}
The morphism $a$ is sometimes called a \emph{Galois alteration}. One
may further assume that $Z'-a^{-1}(X)$ is contained in a strict
normal crossing divisor of $Z'$. We don't need this more precise
form.

By the inductive assumption, it suffices to show the assertion for
$(V,G)$. By \ref{l.uh}, we may replace $(V,G)$ by $(U,G)$, where
$U=V'/H$. We have
\begin{align*}
\Tr(s,H^*_{c}(U,\Q_\ell)) &= \frac{1}{(H:H_0)}\sum \Tr(s',H^*_{c}(V',\Q_\ell)),\\
\Tr(s,H^*_{c,\rig}(U/K)) &= \frac{1}{(H:H_0)}\sum \Tr(s',H^*_{c,\rig}(V'/K)),
\end{align*}
the sums being extended to the classes modulo $H_0$ of elements $s'
\in G'$ above $s$. We may therefore replace
$(U,G)$ by $(V',G')$. By the inductive assumption, we may replace
$(V',G')$ by $(Z',G')$, and we conclude by \ref{5.3}.

By \ref{l.tor}, to prove \ref{5.2} (b), it is enough to show that
$R\Gamma(X/W) \otimes^L_W k$ is a perfect complex of $k[G]$-modules.
We have, by \cite[7.1, 7.24]{BerO},
\[
R\Gamma(X/W) \otimes^L_W k = R\Gamma(X/k) = R\Gamma(X, \Omega^\bullet_{X/k}) =
R\Gamma(Y,\pi_*\Omega^\bullet_{X/k})
\]
in $D^b(k[G])$, where $Y = X/G$ and $\pi \colon X \rightarrow Y$ is the projection
. As $\pi$ is an \'{e}tale Galois cover of
group~$G$, $\pi_*{\cal O}_X$ is, \'{e}tale locally on $Y$,
isomorphic to ${\cal O}_Y[G]$, in particular, is flat over $k[G]$.
The same is true of $\pi_*\Omega^i_{X/k} = \pi_* {\cal O}_X
\otimes_{{\cal O}_Y} \Omega^i_{Y/k}$. For any right $k[G]$-module
$M$, by the projection formula,
\[M\otimes_{k[G]}^L R\Gamma(Y,\pi_*\Omega^i_{X/k})\simeq R\Gamma(Y,M\otimes_{k[G]}(\pi_*\cO_X\otimes_{\cO_Y}\Omega^i_{Y/k}))\]
has cohomology concentrated in $[0,d]$, where $d=\dim(X)=\dim(Y)$.
Thus $R\Gamma(Y,\pi_*\Omega^i_{X/k})$ is a perfect complex of
$k[G]$-modules. Filtering $\pi_*\Omega^{\bullet}_{X/k}$ by the
na\"{\i}ve filtration, we get
\[
\Gr R\Gamma(Y,\pi_*\Omega^{\bullet}_{X/k}) \simeq \bigoplus_{i \in {\mathbb{Z}}} R\Gamma(Y, \pi_*\Omega^i_{X/k}[-i]),
\]
which implies that $R\Gamma(X/k)$ is perfect over $k[G]$ (even as a
filtered complex \cite[V 3.1]{IllCC}).

\begin{Remark}\label{5.4} (a) In the situation of \ref{5.2} (a), assume $X/k$ \emph{proper}.
Then, if $G$ acts freely on~$X$, $\Tr(s,
R\Gamma_{c,\rig}(X/K)) = 0$ for $s\in G$, $s\neq 1$. Indeed, $t_c(s) = 0$ by the
Lefschetz-Verdier trace formula \cite[III]{SGA5}. It seems that so far no general
Lefschetz-Verdier formula is available in rigid cohomology. For
example,  if
$u$ is a fixed point free endomorphism of $X/k$, we don't know
whether $\Tr(u, R\Gamma_{c,\rig}(X/K))= 0$. This vanishing
holds at least in the smooth case (\ref{5.3}).

(b) In the situation of \ref{5.2} (a), it is unknown whether one has
\[
\Tr(s,H^*_{c,\rig}(X/K)) = \Tr(s,H^*_{\rig}(X/K)),
\]
even when $G = \{1\}$.
\end{Remark}

\section{Fixed point sets: around a theorem of P.~Smith}
The results in this section were suggested to the first author by
Serre. They overlap with parts of \cite[\S\S~7, 8]{Se3}.

\begin{Proposition}\label{6.1} \cite[7.2]{Se3}
Let $k$ be an algebraically closed field of characteristic $p$, $\ell$
a prime number $\ne p$, and $X$
an algebraic space separated and of finite type over~$k$, equipped with an
action of an $\ell$-group $G$. Then:
\beq
\chi(X^G) \equiv \chi(X) \mod{\ell}. \label{(6.1.1)}
\eeq
\end{Proposition}

\begin{proof}
By additivity of $\chi$, we may assume $X$ separated. If $G$ is an extension $0 \rightarrow G_1 \rightarrow G \rightarrow
G_2 \rightarrow 0$, $X^{G_1}$ is $G$-stable, $G$ acts on it through
$G_2,$ and $X^G =
(X^{G_1})^{G_2}$. So by induction we may assume that $G$ is cyclic of
order $\ell$. Then $G$ acts freely on $U := X - X^G$, hence, by
\ref{3.2}, $\chi(U) = \ell \chi(U/G)$, and \eqref{(6.1.1)} follows by additivity of
$\chi$. One could also deduce \eqref{(6.1.1)} from \ref{4.9}: $t(g) = \chi(X^G)$
for all $g \ne 1$, and $\sum_{g \in G}t(g) \equiv 0 \mod{\ell}$.
\end{proof}

As Serre observes in (\loccit), \ref{6.1} implies that if $\chi(X)$
is not divisible by~$\ell$, then $X^G$ is not empty. If $X$ is the
affine space ${\mathbb{A}}^n_k$, we find $\chi(X^G) \equiv 1
\mod{\ell}$. In this case, Smith's theory, as recalled in (\loccit)
gives more.

\begin{ssect}\label{6.2} Let $k$ be an algebraically closed field of characteristic $p$, $X$
an algebraic space separated and of finite type over $k$, and let $\ell$ be a
prime number, possibly
equal to~$p$. We say that $X$ is mod $\ell$ \emph{acyclic} if $H^i(X,{\mathbb{F}}_{\ell}) = 0$ for $i \ne 0$ and $H^0(X,{\mathbb{F}}_{\ell}) = {\mathbb{F}}_{\ell}$. When $\ell$ is different from~$p$, by
the finiteness of $H^*(X,{\mathbb{Z}}_{\ell})$ and the exact sequence of
universal coefficients, $X$ is mod~$\ell$ acyclic if and only if
$H^i(X,{\mathbb{Z}}_{\ell}) = 0$ for $i \ne 0$ and $H^0(X,{\mathbb{Z}}_{\ell}) = {\mathbb{Z}}_{\ell}$. When $\ell = p$, if $X/k$ is proper,
connected, and $H^i(X,{\cal O}) = 0$ for $i >0$, then, by the
Artin-Schreier exact sequence, $X$ is mod $\ell$ acyclic. The
following (for schemes) is \cite[7.5 b)]{Se3}.  This result and \ref{6.6}
below (for schemes) were obtained independently by Morin \cite[Th.~2.46]{M}, assuming $G = {\mathbb{F}}_{\ell}$, $\ell\neq p$ and $X^G$ contained in an
affine open subset. His method is similar to ours and is based on a variant \`{a} la
Tate of equivariant cohomology.
\end{ssect}

\begin{Theorem}\label{6.3}
Let $X/k$ and $\ell$ be as in \ref{6.2}. Assume that an $\ell$-group $G$
acts on $X/k$ and that $X$ is mod $\ell$ acyclic. Then so
is $X^G$.
\end{Theorem}

Here is a proof using equivariant cohomology, as in \cite{Bo}. By
d\'{e}vissage, as in the proof of \ref{6.1}, we may assume that $G$ is a
cyclic group of order $\ell$. If $S/k$ is an algebraic space acted on by $G$
and $F$ a $G$-${\mathbb{F}}_{\ell}$-sheaf on $S$, we denote by
$R\Gamma_G(S,F)$ the complex $R\Gamma([S/G],F)$ (where $[S/G]$ is the
associated Deligne-Mumford stack with its \'{e}tale topology), which
can be calculated as
$R\Gamma(S_{\bullet},F_{\bullet})$, where $S_{\bullet}=\cosk_0(S/[S/G])$ is the simplicial algebraic space
defined by the action of $G$ on $S$ and $F_{\bullet}$ the corresponding
simplicial sheaf on $S_{\bullet}$. We have $R\Gamma(S,F) \in D^+({\mathbb{F}}_{\ell}[G])$ and $R\Gamma_G(S,F) = R\Gamma(G,R\Gamma(S,F))$. When
$F$ is the constant sheaf ${\mathbb{F}}_{\ell}$, we write $R\Gamma_G(S)$ for $R\Gamma_G(S,F)$. The
projection $S \rightarrow \Spec k$ makes $H^*_G(S) =
\oplus_{i \ge 0} H^i_G(S)$ into a graded algebra over the graded ${\mathbb{F}}_{\ell}$-algebra
\[
R = H^*_G(\Spec k) = H^*(G,{\mathbb{F}}_{\ell}),
\]
and $H^*_G(S,F)$ into a graded module over $R$.

\begin{sssect}\label{sss}\stepcounter{equation}
Recall that when $\ell = 2$, $R$ is a polynomial algebra
${\mathbb{F}}_{\ell}[x]$ in one generator of degree 1, and when
$\ell > 2$, $R$ is the graded tensor product of the algebra of dual
numbers ${\mathbb{F}}_{\ell}[x]/(x^2)$ with $x$ of degree 1 by a
polynomial algebra ${\mathbb{F}}_{\ell}[y]$ with $y$ of degree~$2$
\cite[XII 7]{CE}.
\end{sssect}

Let $Y = X^G$, $U = X-Y$, $u \colon U \rightarrow X$ the inclusion.
The (equivariant) short exact sequence $0 \rightarrow
u_!{\mathbb{F}}_{\ell}{}_U \rightarrow {\mathbb{F}}_{\ell}{}_X
\rightarrow {\mathbb{F}}_{\ell}{}_{Y} \rightarrow 0$ gives a long
exact sequence of equivariant cohomology \beq\label{(6.3.1)} \dots
\rightarrow H^*_G(X,u_!{\mathbb{F}}_{\ell}{}) \rightarrow H^{*}_G(X)
\rightarrow H^{*}_G(Y) \rightarrow
H^{*+1}_G(X,u_!{\mathbb{F}}_{\ell}{}) \rightarrow \dotsb, \eeq where
$H^*_G = \oplus_{i}H^i_G$. This is an exact sequence of graded
$R$-modules. Consider the commutative diagram \eqref{(3.12.1)}. As
$G$ is cyclic of order $\ell$, $G$ acts freely on $U$, so
$(f_U)_*{\mathbb{F}}_{\ell}{}_U$ is locally free of rank one over
${\mathbb{F}}_{\ell}[G]$, hence
$R\Gamma(G,v_!(f_U)_*{\mathbb{F}}_{\ell}{}_U) =
v_!(\mathbb{F}_{\ell})_{U/G}$. Therefore
\begin{multline}
R\Gamma(G,R\Gamma(X,u_!{\mathbb{F}}_{\ell})) =
R\Gamma(G,R\Gamma(X/G,f_*u_!\mathbb{F}_\ell)) =
R\Gamma(X/G,R\Gamma(G,f_*u_!\mathbb{F}_\ell))\\
= R\Gamma(X/G,R\Gamma(G,v_!(f_U)_*{\mathbb{F}}_{\ell})) =
R\Gamma(X/G,v_!{\mathbb{F}}_{\ell}), \label{(6.3.2)}
\end{multline}
so that
\[
H^*_G(X,u_!{\mathbb{F}}_{\ell}) = H^*(X/G,v_!{\mathbb{F}}_{\ell})
\]
is a graded module of bounded degree, as $\cd_{\ell}(X/G)$
is finite by Lemma \ref{6.3.5} below (whether $\ell$ is different from~$p$ or not). As
$R\Gamma(X,{\mathbb{F}}_{\ell}) = {\mathbb{F}}_{\ell}$, we have $H^*_G(X) = R$.
Therefore, in \eqref{(6.3.1)} the map $H^*_G(X,u_!{\mathbb{F}}_{\ell}{}_U) \rightarrow
H^{*}_G(X)$ vanishes, and \eqref{(6.3.1)} boils down to a short exact
sequence
\[
0  \rightarrow H^{*}_G(X)
\rightarrow H^{*}_G(Y) \rightarrow H^{*+1}(X/G,v_!{\mathbb{F}}_{\ell})
\rightarrow 0,
\]
which can be rewritten
\begin{equation}
0 \rightarrow R \rightarrow
R \otimes_{{\mathbb{F}}_{\ell}} H^*(Y) \rightarrow
H^{*+1}(X/G,v_!{\mathbb{F}}_{\ell}) \rightarrow 0, \label{(6.3.3)}
\end{equation}
since, by K\"unneth's formula for $BG\times Y$,
\begin{equation}\label{e.Kun}
H^{*}_G(Y) = H^*_G \otimes_{{\mathbb{F}}_{\ell}}
H^*(Y) = R \otimes_{{\mathbb{F}}_{\ell}} H^*(Y).
\end{equation}
By Lemma \ref {6.3.4} below, this implies that $H^*(Y)$ is free of rank
1 over ${\mathbb{F}}_{\ell}$, hence $Y$ is mod~$\ell$ acyclic.

\begin{Lemma}\label{6.3.5}
Let $X$ be a separated algebraic space of dimension $d$ of finite type over $k$. Then the $\ell$-cohomological dimension $\cd_\ell X$ is at most $2d$ (\resp $d$) if $\ell\ne p$ (\resp $\ell=p$).
\end{Lemma}

We prove this lemma by induction on $d$. By Chow's lemma \cite[IV.3.1]{Knutson}, we can choose
\[\xymatrix{&X'\ar[d]^\pi\\
U\ar@{^{(}->}[ru]^{j'}\ar@{^{(}->}[r]^j & X}\]
where $X'$ is a scheme, $\pi$ is proper and is an isomorphism over a dense open subscheme $U$ of $X$. Let $\cF$ be an $\ell$-torsion sheaf on $X$. Considering the short exact sequence
\[0\to j_!j^*\cF \to \cF \to Q \to 0\]
and applying the induction hypothesis to $Q$, it is enough to show that for any $\ell$-torsion sheaf $\cG$ on $U$, we have
\[H^i(X,j_!\cG)=0, \text{ $i>2d$ (\resp $i>d$).}\]
Since $\pi$ is proper, we have $j_!\cG \simeq R\pi_* j'_!\cG$, so we have $H^i(X,j_!\cG)=H^i(X',j'_!\cG)$. Hence the lemma follows from the scheme case, which is well-known \cite[X 4.3]{SGA4} (\resp \cite[X 5.2]{SGA4}).

\begin{Lemma}\label{6.3.4}
Let $0 \rightarrow L \rightarrow M \rightarrow N \rightarrow 0$ be an
exact sequence of graded $R$-modules, with $L$ and $M$ free,
and $N$ of bounded degree. Then $L$ and $M$
have the same rank over~$R$.
\end{Lemma}

Indeed, if $\ell = 2$ (resp. $\ell > 2$), $N \otimes_R R[x^{-1}] =
0$ (resp. $N \otimes_R R[y^{-1}] = 0$).

\begin{Remark}
For $X$ and $\ell$ as in \ref{6.2}, $X$ is mod~$\ell$ acyclic if and only if
\[\sum_i\dim H^i(X,\F_\ell)=1.\]
One can consider the analogous condition \beq\label{e.Hc}
\sum_i\dim H^i_c(X,\F_\ell)=1.
\eeq
If $\ell\neq p$, \eqref{e.Hc} is equivalent to $R\Gamma_c(X,\F_\ell)\simeq\F_\ell[-2d]$ and to $R\Gamma_c(X,\Z_\ell)\simeq\Z_\ell[-2d]$, which imply that $X$ is irreducible. Here $d=\dim X$. If $\ell\neq p$ and $X$ is smooth over $k$, then \eqref{e.Hc} is equivalent to $X$ being mod~$\ell$ acyclic by Poincar\'e duality.

For an arbitrary $\ell$, assume that an $\ell$-group $G$ acts on $X$. Using arguments similar to the proof of \ref{6.3}, one can show that \eqref{e.Hc} implies $\sum_i \dim H^i_c(X^G,\F_\ell)=1$. The case $\ell\neq p$ was obtained by Symonds \cite[4.3]{Symonds}, using a different method based on the theorem of Rickard mentioned in \ref{3.11} (c).
\end{Remark}

\begin{Corollary}\label{6.4}
Let $X$ and $\ell$ be as in \ref{6.2}. Assume that $X$ is mod $\ell$
acyclic, and that $X$ is endowed with an
action of a finite group $G$ by $k$-automorphisms.
Then $X/G$
is mod~$\ell$ acyclic.
\end{Corollary}

\begin{proof}
Let $f\colon X\to Y=X/G$ be the projection. Let us first show
\begin{equation}\label{e.Ginv}
(f_*\F_\ell)^G\simeq \F_\ell.
\end{equation}
Let $Y'\to Y$ be an \'etale morphism of finite type. We have
$(f_*\F_\ell)(Y)\simeq\F_\ell^{\pi_0(X')}$, where $X'=X\times_Y Y'$.
Since the projection $X'\to Y'$ identifies $Y'$ with the quotient of
$X'$ by $G$,
\[(f_*\F_\ell)^G(Y')\simeq
\F_\ell^{\pi_0(X')/G}\simeq \F_\ell^{\pi_0(Y')}\simeq \F_\ell(Y').\]

(a) Case where $G$ is an $\ell$-group. If $H$ is a normal
subgroup of $G$, $G$ acts on $X/H$ through $G/H$ and $X/G =
(X/H)/(G/H)$. Therefore we may assume $G$ cyclic of order $\ell$. In
this case, by \ref{6.3} the restriction map
$R\Gamma(X,{\mathbb{F}}_{\ell}) \rightarrow
R\Gamma(X^G,{\mathbb{F}}_{\ell})$ is an isomorphism, hence
$R\Gamma(X,u_!{\mathbb{F}}_{\ell}) = 0$ with the notations of
\eqref{(6.3.1)}. Therefore, by \eqref{(6.3.2)}
$R\Gamma(X/G,v_!{\mathbb{F}}_{\ell}) = 0$, hence the restriction map
$R\Gamma(X/G,{\mathbb{F}}_{\ell}) \rightarrow
R\Gamma(X^G,{\mathbb{F}}_{\ell})$ is an isomorphism. Finally, by
\ref{6.3}, $R\Gamma(X^G,{\mathbb{F}}_{\ell})= {\mathbb{F}}_{\ell}$.

(b) General case. Let $L$ be an $\ell$-Sylow subgroup of $G$. In
order to relate $H^*(X/G,\F_\ell)$ to $H^*(X/L,\F_\ell)$, we
consider the commutative square
\[\xymatrix{X \wedge^L G \ar[r]^h\ar[d]^{f'} & X\wedge^G G\ar[d]^f\\
X\wedge^L \{1\}\ar[r]^g & X\wedge^G \{1\}}
\]
See \ref{3.12} for the
definition of the contracted products.
Since $h$ is a finite \'etale cover of fibers isomorphic to $G/L$,
hence of degree $d=(G:L)$ prime to~$\ell$, the composition
\[(\F_\ell)_X \to h_*(\F_\ell)_{X\wedge^L G} \to (\F_\ell)_X\]
of the adjunction map and the trace map is multiplication by $d$ 
\cite[IX 5.1.4]{SGA4}. Applying $f_*$ and taking $G$-invariants, we
see that the composition
\[(f_*(\F_\ell)_X)^G \to (f_* h_*(\F_\ell)_{X\wedge^L G})^G \to (f_*(\F_\ell)_X)^G\]
is again multiplication by $d$. Hence $(\F_\ell)_{X/G} \simeq
(f_*(\F_\ell)_X)^G$ \eqref{e.Ginv} is a direct factor of
\[g_*(\F_\ell)_{X/L}
\simeq g_*(f'_*(\F_\ell)_{X\wedge^L G})^G \simeq (f_* h_*
(\F_\ell)_{X\wedge^L G})^G.
\]
It follows that $g$ induces an injection $H^*(X/G,\F_\ell)\to
H^*(X/L, \F_\ell)$ and therefore an isomorphism
$H^*(X/G,\F_\ell)\simeq \F_\ell$ by case (a) above.
\end{proof}

There are many variants and generalizations of
\ref{6.3}. Here are two of them.

\begin{Theorem}\label{6.5} (\cf \cite[7.5 a)]{Se3})
Let $X/k$ and $\ell$ be as in \ref{6.2}. Assume that an $\ell$-group
$G$ acts on $X/k$. Let $N$ be an integer such that
$H^i(X,{\mathbb{F}}_{\ell}) = 0$ for $i > N$. Then
$H^i(X^G,{\mathbb{F}}_{\ell}) = 0$ for $i > N$.
\end{Theorem}

We may assume $G$ cyclic of order $\ell$. Serre's proof of \ref{6.3}
and \ref{6.5} (\loccit) makes no use of equivariant cohomology, but
instead exploits the action of $\F_\ell[G]$ on
$f_*{\mathbb{F}}_{\ell}$, with the notations above. It is also easy
to prove \ref{6.5} along the lines of the above proof of \ref{6.3}.
Again, the key point is that, by \ref{6.3.5} applied to $Y=X^G$, the
restriction homomorphism
\[
r \colon H^*_G(X) \rightarrow H^*_G(Y)=R
\otimes_{{\mathbb{F}}_{\ell}} H^*(Y)
\]
\eqref{e.Kun} is a TN-isomorphism of graded $R$-modules (in the
sense of \cite[3.4]{EGAII}), \ie there exists an integer $n_0$ such
that $r_n \colon H^n_G(X) \rightarrow H^n_G(Y)$ is an isomorphism
for $n \ge n_0$. The source and target of $r$ are the abutments of
spectral sequences
\begin{align*}
E(X) \mathpunct{:} E^{ij}_2 = H^i(G,H^j(X)) &\Rightarrow
H^{i+j}_G(X),
\\
E(Y) \mathpunct{:} E^{ij}_2 = H^i(G,H^j(Y)) &\Rightarrow
H^{i+j}_G(Y),
\end{align*}
the second one being degenerate at $E_2$, and $r$ underlies a morphism of spectral sequences $E(X) \rightarrow
E(Y)$. Take $n \ge \sup(n_0,N)$. Then $H^n_G(X) = F^{n-N}H^n_G(X)$
(where $F^{\bullet}$ denotes the filtration on the abutment). As
$r_n$ is a filtered isomorphism, this implies that $H^n_G(Y) =
F^{n-N}H^n_G(Y)$, and consequently that $H^j(Y) = 0$ for $N < j \le
n$.

\begin{Theorem}\label{6.6}
Let $k$ and $\ell$ be as in \ref{6.2}, $G$ be an $\ell$-group,
$f\colon X\to X'$ be a $G$-equivariant morphism of algebraic
$k$-spaces which are separated of finite type. We denote by
$R\Gamma(X'/X,\F_\ell)$ the complex defining the relative cohomology
of $X'$ modulo $X$. Assume that
$R\Gamma(X'/X,\F_\ell)\simeq\F_\ell[-N]$. Then there exists $M\le N$
such that $\ell(N-M)$ is even and $R\Gamma(X'^G/X^G,\F_\ell)\simeq
\F_\ell[-M]$.
\end{Theorem}

Relative cohomology is defined in \cite[6.3]{HodgeIII} and more
generally in \cite[III 4.10]{IllCC}. We have an exact triangle
\[R\Gamma(X'/X,\F_\ell)\to R\Gamma(X',\F_\ell)\to R\Gamma(X,\F_\ell)\to.\]

\begin{proof}
We may assume that $G$ is cyclic of order~$\ell$. Let $u\colon
X-X^G\hookrightarrow X$, $u'\colon X'-X'^G\hookrightarrow X'$. The
first line of the 9-diagram
\[\xymatrix{C\ar[r]\ar[d] & R\Gamma(X'/X)\ar[r]\ar[d] & R\Gamma(X'^G/X^G)\ar[r]\ar[d] &\\
R\Gamma(X',u'_!\F_\ell)\ar[r]\ar[d] & R\Gamma(X')\ar[r]\ar[d] & R\Gamma(X'^G)\ar[r]\ar[d]&\\
R\Gamma(X,u_!\F_\ell)\ar[r]\ar[d] & R\Gamma(X)\ar[r]\ar[d] & R\Gamma(X^G)\ar[r]\ar[d]&\\
&&}
\]
gives rise to a long exact sequence
\begin{equation}\label{e.es}
\dots \rightarrow H^*(G,C) \rightarrow  H^{*}_G(X'/X)
\rightarrow H^{*}_G(X'^G/X^G) \rightarrow
H^{*+1}(G,C) \rightarrow \dotsb,
\end{equation}
where $H^{*}_G(X'/X)=H^*(G,R\Gamma(X'/X))$,
$H^{*}_G(X'^G/X^G)=H^*(G,R\Gamma(X'^G/X^G))$. By \eqref{(6.3.2)} and
\ref{6.3.5}, $H^*(G,C)$ is of bounded degree. We have a spectral
sequence
\[
 E(X'/X) \mathpunct{:} E^{ij}_2 = H^i(G, H^j(X'/X)) \Rightarrow
 H^{i+j}_G(X'/X),
\]
analogous to $E(X)$, whose $E_2$ term is concentrated on the
horizontal line of degree~$N$, and therefore degenerates at $E_2$,
yielding isomorphisms $E^{i,N}_2 \longsimto H^{i+N}_G(X'/X)$. As $G$
can't act on ${\mathbb{F}}_{\ell}$ but trivially, we get $
H^*_G(X'/X) \simeq R \otimes H^N(X'/X)$, \ie $H^*_G(X'/X) \simeq R(-N)$,
where $(-)$ is the usual shift on graded modules. Then \eqref{e.es}
boils down to an exact sequence
\[
0 \rightarrow R(-N) \rightarrow R \otimes_{{\mathbb{F}}_{\ell}}
H^*(X'^G/X^G) \rightarrow
H^{*+1}(G,C) \rightarrow 0,
\]
which shows that $R \otimes H^*(X'^G/X^G)$ is (graded) free of rank
one on $R$, and therefore that $H^*(X'^G/X^G)$ is concentrated in one
degree $M\le N$, and of dimension one over ${\mathbb{F}}_{\ell}$. If
$\ell>2$, as the graded pieces of $R$ with odd degrees are killed by
$x$ (\ref{sss}), $N-M$ must be even.
\end{proof}

Let $Y$ be an algebraic subspace of $X$. We say the pair $(X,Y)$ is a
\emph{mod $\ell$ cohomology $N$-disk} if $R\Gamma(X/Y,\F_\ell)\simeq
\F_\ell[-N]$.

\begin{Corollary}
Let $X/k$ and $\ell$ be as in \ref{6.2}, $G$ be an $\ell$-group
acting on~$X$, $Y$ be a $G$-equivariant algebraic subspace of $X$.
Assume that $(X,Y)$ is a mod $\ell$ cohomology $N$-disk for some
$N\ge 0$. Then $(X^G,Y^G)$ is a mod $\ell$ cohomology $M$-disk for
some $0\le M\le N$ with $\ell(N-M)$ even.
\end{Corollary}

Let $X$ be an algebraic $k$-space separated of finite type. We write
\[R\tilde \Gamma(X,\F_\ell)=R\Gamma(\Spec k/X,\F_\ell)[1]\in D^{\ge -1}.\]
It is a cone of the adjunction morphism $\F_\ell=R\Gamma(\Spec
k,\F_\ell)\to R\Gamma(X,\F_\ell)$. We say that $X$ is a \emph{mod
$\ell$ cohomology $N$-sphere} if $R\Gamma(X,\F_\ell)\simeq
\F_\ell[-N]$. For $N=-1$ (resp.\ $N=0$, resp.\ $N\ge 1$), $X$ is a
mod $\ell$ cohomology $N$-sphere if and only $X$ is empty (resp.\
$R\Gamma(X,\F_\ell)\simeq \F_\ell\oplus \F_\ell$, resp.\
$H^0(X,\F_\ell)\simeq H^N(X,\F_\ell)\simeq\F_\ell$ and
$H^q(X,\F_\ell)=0$ for $q\neq 0,N$). Applying \ref{6.6} to the
structure morphism $X\to \Spec k$, we obtain the following.

\begin{Corollary}\label{c.1}
Let $X/k$ and $\ell$ be as in \ref{6.2}. Assume that an $\ell$-group
$G$ acts on~$X$ and that $X$ is a mod $\ell$ cohomology $N$-sphere
for some $N \ge -1$. Then $X^G$ is a mod $\ell$ cohomology
$M$-sphere with $-1\le M\le N$ and $\ell(N-M)$ even.
\end{Corollary}

These corollaries are analogues of \cite[III 5.1, 5.2]{Br}. \ref{c.1}
is an analogue of the main result in \cite{Bo}.

\section{A localization theorem}

In this section we fix a prime number $\ell$ and an algebraically
closed field $k$ of characteristic $p$.

In the situation of \ref{6.3}, with $G = {\mathbb{F}}_{\ell}$, the
restriction map $H^*_G(X) \rightarrow H^*_G(X^G)$ induces an
isomorphism of graded $R$-algebras $H^*_G(X)[y^{-1}] \longsimto
H^*_G(X^G)[y^{-1}]$ (resp. $H^*_G(X)[x^{-1}] \longsimto
H^*_G(X^G)[x^{-1}]$) for $\ell > 2$ (resp. $\ell = 2$). We give a
generalization in this section.

\begin{ssect}\label{7.1} Following Quillen's terminology \cite[\S~4]{Q1}, for $r \in {\mathbb{N}}$, by an \emph{elementary abelian $\ell$-group of rank $r$} we mean a group $G$
isomorphic to the direct product of $r$ cyclic groups of order
$\ell$, \ie the group underlying a vector space of dimension $r$
over ${\mathbb{F}}_{\ell}$ (such groups are sometimes called groups of
type $(\ell, \ell, \dots, \ell)$).  Consider the short exact sequence
\[0\to \F_\ell \to \Z/\ell^2 \Z \to \F_\ell \to 0,\]
with trivial $G$ actions. It induces an exact sequence
\[\Hom(G,\F_\ell)\xrightarrow{\sim} \Hom(G,\Z/\ell^2 \Z) \to \Hom(G,\F_\ell)\xrightarrow{\beta} H^2(G,\F_\ell),\]
where we have identified $H^1(G,{\mathbb{F}}_{\ell})$ with $\check G = \Hom(G,{\mathbb{F}}_{\ell})$ by the natural isomorphism.
It follows that the Bockstein operator $\beta$ is injective.
Recall (\loccit) that we have a natural identification of ${\mathbb{F}}_{\ell}$-graded
algebras
\[
H^*(G,{\mathbb{F}}_{\ell}) = \begin{cases}S(\check G) & \text{if $\ell = 2$} \\
\Lambda(\check G) \otimes S(\beta \check G)  & \text{if $\ell >
2$}\end{cases}\label{(7.1.1)}
\]
where $S$ (resp. $\Lambda$) denotes a symmetric (resp. exterior)
algebra over ${\mathbb{F}}_{\ell}$. In particular, if $\{x_1,\dots, x_r\}$
is a basis of $\check G$ over ${\mathbb{F}}_{\ell}$, then
\[
H^*(G,{\mathbb{F}}_{\ell}) = \begin{cases}{\mathbb{F}}_{\ell}[x_1,\dots,x_r] & \text{if $\ell = 2$} \\
\Lambda(x_1,\dots,x_r) \otimes {\mathbb{F}}_{\ell}[y_1,\dots,y_r]  & \text{if $\ell > 2$}\end{cases}\label{(7.1.2)}
\]
where $y_i = \beta x_i$.
\end{ssect}

We put $V_G=\Spec(S(\beta \check G))$. It is an affine space of
dimension $r$ over $\F_\ell$. For any subgroup $H$ of $G$, the
surjection $\check G\to \check H$ induces a closed immersion
$V_H\hookrightarrow V_G$. For $f\in S(\beta \check G)$, we denote by
$V_G(f)$ the closed subset of $V_G$ defined by $f$. Then
\[V_H=\bigcap_{x\in \Ker(\check G\to \check H)} V_G(\beta x).\]
For any $S(\beta \check G)$-module $M$, we denote by
$\Supp_G(M)\subset V_G$ the support of $M$ and by $\Suppc_G(M)$ the
Zariski closure of $\Supp_G(M)$.

The following is an analogue of the localization theorem of
Borel-Atiyah-Segal for actions of tori \cite[6.2]{GKM}. As in
\ref{3.12}, for any algebraic space $X$ separated of finite type over
$k$, endowed with an action of $G$, and any subgroup $H$ of $G$, we
put $X_H=X^H-\bigcup_{H'} X^{H'}$, where $H'$ runs over subgroups of
$G$ strictly containing $H$. Since $G$ is abelian, $G$ acts on $X^H$
and $X_H$. As in \ref{6.2} we do not assume $\ell$ invertible in $k$.

\begin{Theorem}\label{7.2}
Let $G$ be an elementary abelian $\ell$-group $G$, $Y\to X$ be a
$G$-equivariant closed immersion between algebraic spaces separated
and of finite type over~$k$. Let $j\colon U=X-Y\hookrightarrow X$ be
the complementary open immersion, $\fH$ be the set of subgroups $H$
of $G$ such that $R\Gamma(X/G,v_{H!}\F_\ell)\neq 0$, $v_{H}\colon
U_H/G\hookrightarrow X/G$ is the inclusion. Let $T= \bigcup_{H\in
\fH} V_H\subset V_G$.
\begin{enumerate}
\item We have
\[\Suppc_G(H^*_G(X,j_!\F_{\ell U}))=T.
\]

\item If $T\neq V_{\{0\}}$, then
    \[\Suppc_G(H^{2*}_G(X,j_!\F_{\ell U}))=T.\]

\item For any $e\in S(\beta \check G)$ such that $V_G(e)\supset
    T$, the restriction map $H^*_G(X,\F_\ell) \rightarrow
    H^*_G(Y,\F_\ell)$ induces an isomorphism
\[
H^*_G(X,\F_\ell)[e^{-1}] \longsimto
H^*_G(Y,\F_\ell)[e^{-1}]
\]
of graded $H^*(G,\F_\ell)[e^{-1}]$-algebras.
\end{enumerate}
\end{Theorem}

Note that $\fH$ is a subset of the set $\fH'$ of subgroups $H$ of $G$
such that $U_H\neq\emptyset$. For any maximal element $H$ of $\fH'$,
$U_H=U^H$ is closed in $U$. Therefore, if $Y=\emptyset$, then $\fH$
and $\fH'$ have the same maximal elements, so that $T=\bigcup_{H\in
\fH'}V_H$.

\begin{proof}
By the long exact sequence of equivariant cohomology
\begin{equation}\label{(e.7.2)}
\dots\to
H^*_G(X,j_!\F_{\ell U})\to H^*_G(X)\to H^*_G(Y)\to \dotsb,
\end{equation}
(c) is equivalent to saying
\[
H^*_G(X,j_!\F_{\ell U})[e^{-1}]=0,
\]
which is a consequence of (a).

For every subgroup $H$ of $G$, we denote by $j_H\colon
U_H\hookrightarrow X$ the immersion. We will show the following.
\begin{itemize}
\item[(a')]
\[
\Supp_G (H^*_G(X,j_{H!}(\F_\ell)_{U_H}))=
\begin{cases}
V_H&\text{if $H\in \fH$,}\\
\emptyset&\text{if $H\not\in \fH$.}
\end{cases}
\]

\item[(b')] For $H\in \fH$ satisfying $H\neq \{0\}$, we have
    \[\Supp_G(H^{2*}_G(X,j_{H!}(\F_\ell)_{U_H}))=V_H.\]
\end{itemize}
Let us first prove that (a') and (b') imply (a) and (b). Note that
$j_!(\F_{\ell U})$ is a successive extension of
$j_{H!}(\F_\ell)_{U_H}$, $H$ running through subgroups $H$ of $G$.
Thus (a') implies
\begin{equation}\label{(7.2.i)}
  \Supp_G(H^*_G(X,j_!\F_{\ell U}))\subset \bigcup_{H}\Supp_G (H^*_G(X,j_{H!}\F_\ell)) = T.
\end{equation}
For any $H$, we denote by $\fp_H$ the generic point of $V_H$,
considered as a prime ideal of $S(\beta \check G)$. The generic
points of $T$ are $\fp_H$, $H$ running through maximal elements of
$\fH$. For any such $H$, (a') implies
\[H^*_G(X,j_{H'!}\F_\ell)_{\fp_{H}}=0\]
for all $H'\neq H$. Thus
\[H^*_G(X,j_!\F_{\ell U})_{\fp_{H}}\simeq H^*_G(X,j_{H!}\F_\ell)_{\fp_{H}}\neq 0.\]
In other words, $\Supp_G(H^*_G(X,j_!\F_{\ell U}))$ contains all
generic points of $T$. Combing this with \eqref{(7.2.i)}, we obtain
(a). If $T\neq V_{\{0\}}$, the origin of $V_G$ is not a generic point
of $T$. Then, as above, one deduces from (b') that
$\Supp_G(H^{2*}_G(X,j_!\F_{\ell U}))$ contains all generic points of
$T$, which proves (b).

To show (a') and (b'), choose a subgroup $H'$ of $G$ such that
$G=H\oplus H'$ and consider the Cartesian square
\[
\xymatrix{U_{H} \ar[r]^{u} \ar[d]^{f'} & X^H \ar[d]^f\\
U_{H}/H' \ar[r]^{v} & X^H/H'}
\]
We have $H^*(G)\simeq H^*(H)\otimes_{\F_\ell} H^*(H')$ and
isomorphisms of $H^*(G)$-modules
\begin{equation}\label{(7.2.3)}
H^*_G(X,j_{H!}\F_\ell)= H^*_G(X^H,u_!\F_\ell)\simeq H^*(H)\otimes_{\F_\ell} H^*_{H'}(X^H,u_!\F_\ell).
\end{equation}
Here $\F_\ell=(\F_\ell)_{U_H}$. Since $f'$ is a Galois \'{e}tale
cover of group $H'$, we have
\begin{equation}\label{(7.2.4)}
R\Gamma_{H'}(X^H,u_{!}{\mathbb{F}}_{\ell}) \simeq
R\Gamma_{H'}(X^H/H',f_*u_!{\mathbb{F}}_{\ell})) \simeq R\Gamma(X^H/H',R\Gamma(H',v_!f'_*{\mathbb{F}}_{\ell}))\simeq R\Gamma(X^H/H',v_!\F_\ell).
\end{equation}
As $\cd_{\ell}(X^H/H')$ is finite (\ref{6.3.5}), it follows that
$H^*_{H'}(X^H,u_!{\mathbb{F}}_{\ell})$ is of bounded degree. Thus,
for every $y\in {H'}\spcheck$,
\[
  H^*_{H'}(X^H,u_!\F_\ell)[(\beta y)^{-1}]=0.
\]
It follows that
\[\Supp_{H'}(H^*_{H'}(X^H,u_!\F_\ell))=
\begin{cases}
V_{\{0\}}&\text{if $H\in \fH$,}\\
\emptyset&\text{if $H\not\in \fH$,}
\end{cases}\]
which implies (a') by \ref{l.supp} (a) below applied to the
projection $V_G\to V_{H'}$. Now let $H$ be a nonzero element of
$\fH$. For nonzero elements $x\in \check H$ and $\alpha\in
H^i_{H'}(X^H,u_!\F_\ell)$, $1\otimes \alpha$ and $x\otimes \alpha$
in the right-hand side of \eqref{(7.2.3)} generate free sub-$S(\beta
\check H)$-modules of rank 1 of $H^{2*+i}_G(X,j_{H!}\F_\ell)$ and
$H^{2*+i+1}_G(X,j_{H!}\F_\ell)$, respectively. Thus
\[\Supp_H(H^{2*}_G(X,j_{H!}\F_\ell))=V_H.\]
Since $\Supp_G(H^{2*}_G(X,j_{H!}\F_\ell))\subset V_H$ by (a'), this
implies (b') by \ref{l.supp} (b) applied to the projection $V_G\to
V_H$.
\end{proof}

\begin{Lemma}\label{l.supp}
  Let $f\colon Y\to Z$ be a morphism of schemes.
  \begin{enumerate}
    \item Assume $f$ is flat. Then for any quasi-coherent sheaf
        $\cG$ on $Z$, $\Supp(f^*\cG)=f^{-1}(\Supp(\cG))$.

    \item Assume $f$ is affine. Let $\cF$ be a quasi-coherent
        sheaf $\cF$ on $Y$ of support contained in a subscheme
        $Y_0$ of $Y$ such that $f|Y_0\colon Y_0\to Z$ is
        universally closed. Then $\Supp(f_*\cF)=f(\Supp(\cF))$.
  \end{enumerate}
\end{Lemma}

\begin{proof}
  (a) For any point $y$ of $Y$, since $\cO_{Y,y}$ is faithfully flat
  over $\cO_{Z,f(y)}$,
  \[x\in \Supp(f^*\cG) \Longleftrightarrow \cG_{f(y)}\otimes_{\cO_{Z,f(y)}}\cO_{Y,y} = (f^*\cG)_y\neq 0
  \Longleftrightarrow \cG_{f(y)}\neq 0 \Longleftrightarrow f(y)\in \Supp(\cG).\]

  (b) We may assume $Y=\Spec(A)$, $Z=\Spec(B)$. Let $\fq\in
  \Spec(B)$. Then $\fq\in \Supp(f_*\cF)$, i.e.\ $\cF\otimes_B
  B_{\fq}\neq 0$, if and only if $(\cF\otimes_B B_{\fq})_{\fp}\neq 0$ for
  some maximal ideal $\fp$ of $A\otimes_B B_{\fq}$. By assumption, for any such
  $\fp$, $(f\otimes_B B_{\fq})(\fp)$ is the closed point $\fq$ of
  $\Spec(B_{\fq})$. Thus $\fq\in \Supp(f_*\cF)$ if and only if $\cF_{\fp}\neq
  0$ for some $\fp\in f^{-1}(\fq)$, i.e.\ $\fq\in f(\Supp(\cF))$.
\end{proof}

Applying \ref{7.2} (c) to $j\colon X-X^G\hookrightarrow X$, we obtain
the following analogue of Quillen's localization theorem
\cite[4.2]{Q1}.

\begin{Corollary}\label{7.c}
Let $X$ be an algebraic space separated and of finite type over~$k$,
endowed with an action of an elementary abelian $\ell$-group $G$ of
rank $r$. Let $e=\prod_{x\in \check G-\{0\}}\beta x\in
H^{2(\ell^r-1)}(G,\F_\ell)$. Then the restriction map
$H^*_G(X,\F_\ell) \rightarrow H^*_G(X^G,\F_\ell)$ induces an
isomorphism
\[
H^*_G(X,\F_\ell)[e^{-1}] \longsimto
H^*_G(X^G,\F_\ell)[e^{-1}]
\]
of graded $H^*(G,\F_\ell)[e^{-1}]$-algebras.
\end{Corollary}

\begin{Remark}\label{7.3}
(a) In \ref{7.2}, if $T=V_{\{0\}}$, it may happen that
$H^{2*}_G(X,j_!\F_{\ell U})=0$. In fact, if $G=\{1\}$, $X$ is mod
$\ell$ acyclic (\ref{6.2}), and $Y$ is the disjoint union of $n$
rational points, then $R\Gamma(X,j_!\F_{\ell U})=\F_\ell^{n-1}[-1]$.

(b) In the situation of \ref{7.c}, assume $G$ has rank 1. Since
$H^*_G(X,j_{!}{\mathbb{F}}_{\ell})$ is of bounded degree,
\eqref{(e.7.2)} implies that the map
\[
\rho \colon H^*_G(X) \rightarrow H^*_G \otimes H^0(X^G),
\]
defined by the restriction $H^*_G(X) \rightarrow H^*_G(X^G) = H^*_G
\otimes H^*(X^G)$ \eqref{e.Kun} composed with the projection onto
$H^*_G \otimes H^0(X^G)$, has the following property: there exists
an integer $N$ such that, for any element $z \in \Ker \rho$ (resp.
$z \in H^*_G \otimes H^0(X^G)$) of positive degree, $z^N = 0$ (resp.
$z^N \in \Img \rho$).  In a future paper \cite{IZ}, we will discuss
a generalization of this fact, analogous to Quillen's theorem
\cite[6.2]{Q1}.

(c) In the situation of \ref{7.2}, assume $\ell \neq p$. Then
\eqref{(7.2.3)} and \eqref{(7.2.4)} imply that
$H^*_G(X,j_{H!}\F_\ell)$ is a finitely generated $S(\beta\check
G)$-module, because $H^*(X^H/H',v_!\F_\ell)$ is a finite-dimensional
vector space (cf.\ \ref{d.fin}). It follows that $H^*_G(X,j_!\F_{\ell
U})$ is a finitely generated $S(\beta\check G)$-module and
\[\Suppc_G(H^*_G(X,j_!\F_{\ell
U}))=\Supp_G(H^*_G(X,j_!\F_{\ell U})).
\]
In the future paper, we will prove a finiteness result for more
general groups $G$, analogous to Quillen's finiteness theorem
\cite[2.1]{Q1}.
\end{Remark}

\subsection*{Acknowledgements} This paper grew out of questions of
J.-P.~Serre. We thank him heartily for invaluable comments and
suggestions. We also thank A.~A.~Beilinson for kindly communicating
to us Rouquier's note \cite{Ro}, G.~Laumon for discussions on
equivariant cohomology, M.~Olsson for giving us a proof of
\ref{6.3.5}, A.~Tamagawa for pointing out an error in a previous
formulation of \ref{3.7}, and T.~Saito for suggestions on \S~2. The
second author is grateful to A.~Abbes for an invitation to
l'Universit\'e de Rennes 1, where part of this work was done. We
warmly thank the referee for a very careful reading of the
manuscript and many helpful remarks.


\end{document}